\documentclass{amsart}

\usepackage{amsmath,amssymb,amsthm,graphicx,slashed,subfig}%
\usepackage{amsmath}%
\usepackage{amsfonts}%
\usepackage{amssymb}%

\usepackage{hyperref,enumerate}

\usepackage{colortbl}

\usepackage[all,cmtip]{xy}

\definecolor{Red}{rgb}{0.6,0,0}
\def\Red{\textcolor{black}}

\usepackage{graphicx}
\providecommand{\U}[1]{\protect\rule{.1in}{.1in}}
\newtheorem{thm}{Theorem}[section]
\newtheorem{corl}[thm]{Corollary}

\newtheorem{lma}[thm]{Lemma}
\newtheorem{lem}[thm]{Lemma}
\newtheorem{prop}[thm]{Proposition}
\newtheorem{defn}[thm]{Definition}
\newtheorem{ex}[thm]{Example}
\newtheorem{rem}[thm]{Remark}
\newtheorem{fact}[thm]{Fact}
\def\tilde{\widetilde}
\def\alg{\mathrm{alg}}
\def\A{\mathcal{A}}
\def\cS{\mathcal{S}}
\def\B{\mathcal{B}}
\def\bullet{\cdot}
\def\bar{\overline}
\newcommand{\bra}[1]{\langle #1 |}
\newcommand{\ket}[1]{| #1 \rangle}
\def\C{\mathbb{C}}
\def\cb{\mathrm{cb}}

\newcommand{\CZ}[1]{C^*(\Z)_{(#1)}}

\def\E{\mathcal{E}}
\def\env{\mathrm{env}}
\def\epsilon{\varepsilon}
\def\eps{\varepsilon}

\def\F{\mathcal{F}}

\def\H{\mathcal{H}}

\def\id{\mathbb{I}}
\def\K{\mathcal{K}}

\def\min{\mathrm{min}}
\def\N{\mathbb{N}}

\def\cP{\mathcal P}

\def\phi{\varphi}
\def\R{\mathbb{R}}
\def\cR{\mathcal{R}}

\def\cS{\mathcal{S}}
\def\T{\mathbb{T}}
\DeclareMathOperator{\tr}{Tr}
\newcommand{\Toep}[1]{C(S^1)^{(#1)}}
\def\U{\mathcal{U}}
\def\Z{\mathbb{Z}}
\def\opcit{{\it op.cit.\/}\ }
\def\sext{{$C^{\sharp }$}}
\def\ker{{\mbox{Ker}}}
\def\qqq{\,,\,~\forall}
\newcommand{\ie}{{\it i.e.\/}\ }

\newtheoremstyle{commentstyle}
  {0.2cm}{0.2cm}
  {\sf}
  {0cm}
  {\bfseries}{ }
  {0cm}
  {\thmname{#1}\thmnumber{ #2}:\thmnote{ #3}}

\theoremstyle{commentstyle}
\newtheorem{mycomment}{Comment}

\title[Spectral truncations in NCG and operator systems]{Spectral truncations in noncommutative geometry and operator systems}
\author{Alain Connes}
\address{College de France, 3 rue Ulm, F75005, Paris, France\\
I.H.E.S. F-91440 Bures-sur-Yvette, France\\
Department of Mathematics, Ohio State University, Columbus OH 43210 USA\\}
\email{alain@connes.org}
\author{Walter D. van Suijlekom}

\address{Institute for Mathematics, Astrophysics and Particle Physics, Radboud
University Nijmegen, Heyendaalseweg 135, 6525 AJ Nijmegen, The Netherlands.
}
\email{waltervs@math.ru.nl}
\date{\today}

\begin{document}

\begin{abstract}
In this paper we extend the traditional framework of noncommutative geometry in order to deal with spectral truncations of geometric spaces ({\em i.e.} imposing an ultraviolet cutoff in momentum space) and with tolerance relations which provide a coarse grain approximation of geometric spaces at a finite resolution. In our new approach  the traditional role played by $C^*$-algebras is taken over by operator systems. As part of the techniques we treat $C^*$-envelopes, dual operator systems and stable equivalence. We define a propagation number for operator systems, which we show to be an invariant under stable equivalence and use to compare approximations of the same space. 

We illustrate our methods for concrete examples obtained by spectral truncations of the circle. These are operator systems of finite-dimensional Toeplitz matrices and their dual operator systems which are given by functions in the group algebra on the integers with  support in a fixed interval. It turns out that the cones of positive elements and the pure state spaces for these operator systems possess a very rich structure which we analyze including for the algebraic geometry of the boundary of the positive cone and the metric aspect \ie the distance on the state space associated to the Dirac operator. The main property of  the spectral truncation is that it keeps the isometry group intact. In contrast, if one considers the other finite approximation provided by circulant matrices the isometry group becomes discrete, even though in this case the operator system is a $C^*$-algebra. We analyze this in the context of the finite Fourier transform on the cyclic group. 

The extension of noncommutative geometry to operator systems allows one to deal with metric spaces up to finite resolution by considering the relation $d(x,y)< \epsilon$ between two points, or more generally a tolerance relation  which naturally gives rise to an operator system.

\end{abstract}

\maketitle
\tableofcontents

\section{Introduction}
\label{sect:intro}

Noncommutative geometry \cite{C94} has shown that it is possible to give a fully spectral description of Riemannian spin manifolds. In fact, the mere knowledge of the spectrum of the Dirac operator $D$ relative to that of a function algebra $A$ allows one to reconstruct the full Riemannian spin manifold $M$ \cite{C08}. In physical terms this can be phrased by saying that we can probe the structure of curved spacetime around us by means of eigenfrequencies and eigenfunctions for a fermion that moves through that spacetime.

For the mathematical reconstruction of $M$ it is crucial to know the full spectrum of $D$ and $A$. In practice, however, it is clear that we will only have access to part of that spectrum. Indeed, we are limited by the power and resolution of our detectors and typically study physical phenomena up to a certain energy scale. Motivated by this we pose the following question:
\begin{quote}
{\em   can the framework of noncommutative geometry be extended to the case where only part of the spectrum of $D$ is available together with, say, a certain truncation of the algebra $A$?}
\end{quote}
This question has been present all along in the development of the relation between noncommutative geometry and physics with the long term goal of finding testable models of quantum gravity from truncated versions of the model given by quanta of geometry in \cite{CCM14,CCM15}. It has been clear from the start that spectral truncation, which means introducing a cutoff in momentum space truncating the Hilbert space of fermions, respects all continuous symmetries and is superior to an artificial discretization. The spectral truncation in relation with the metric aspect has been studied and formalized in the work \cite{ALM14} and the present paper is directly in line with this development. Our new input is  to  put forward the role of operator systems in the general theory and to analyze in great details, including the algebraic geometry of the boundary of the positive cone, the example of the truncated circle with its wealth of structure coming from the theory of Toeplitz matrices.

The way operator systems naturally arise in the process of spectral truncation is as follows. 
Since the self-adjoint Dirac operator $D$ acts in a Hilbert space $H$, a natural spectral truncation is given simply by a spectral projection $P$ onto eigenspaces of $D$; this is an operator that commutes with $D$. The natural truncation of the action of the $*$-algebra $A$ to the Hilbert space $PH$ is given by the space $PAP$. Since $P$ does not commute with $A$, this is not an algebra anymore. Moreover in many interesting examples if one takes the $C^*$-algebra generated by $PAP$ one gets a non-informative full matrix algebra. However, not all is lost: the space $PAP$ is a $*$-closed subspace in $\B(H)$, {\em i.e.} it is a so-called {\em operator system} \cite{CE77} ({\em cf.} \cite{ER00,Pau02,Pis03,Bla06}). Such spaces have an extremely rich structure: they are matrix ordered, they possess cones of positive elements, (pure) state spaces, {\em et cetera}. Moreover, for truncations it turns out that symmetries of the pair $(A,D)$ induce symmetries of $(PAP,PDP)$. This follows quite easily from the fact that $P$ is a spectral projection of $D$.

Another viewpoint is obtained when we change our perspective from momentum to position space. The energy cutoff translates to a consideration of metric spaces with a certain finite resolution $\epsilon$, where we say that two points $x,y$ are equivalent if $d(x,y) < \eps$. This is not an equivalence relation, but it is a so-called {\em tolerance relation}.  We may imitate the construction of the $C^*$-algebra of a foliation or more generally of a groupoid $C^*$-algebra  with the crucial difference that the convolution product cannot be defined due to the lack of transitivity. However, given the symmetry and reflexivity, a tolerance relation does  define an operator system. This creates a generalization of the basic construction of noncommutative geometry which started from analysing the geometric examples of intractable spaces of leaves of foliations using noncommutative algebras. Here the issue of the lack of transitivity of the relation is already present in the simplest case where the generated equivalence relation has a single class \ie corresponds to the  $C^*$-algebra of compact operators. Besides the  class of examples of operator systems associated to spectral truncations a whole new class thus appears from tolerance relations. Such relations appear naturally  in the  homotopy theory of simplicial complexes which do not fulfill the Kan-extension property as  was shown in \cite{CCgromov,CCfrob}. 

In this paper we develop the formalism needed for doing noncommutative geometry with operator systems. We will focus mainly on the `topological properties' described by the operator systems and corresponding state spaces. The metric aspect as provided by spectral triples extends in a straightforward manner to the new framework and in \S \ref{sect:trunc-dist} we analyse the distance function on the truncated circle. For some preliminary results on the metric aspect, we refer to \cite{ALM14,Ber19}. Computer simulations involving a spectral truncation adopting also the Heisenberg quantization relation of \cite{CCM14,CCM15} have been reported in \cite{GS19a,GS19b}.

The key concepts that we will discuss and introduce here are:
\begin{itemize}
\item operator systems: both concrete and abstract (in the sense of \cite{CE77});
\item duality between operator systems (in the sense of \cite{Arv69,CE77});
\item enveloping $C^*$-algebras of operator systems (in the sense of \cite{Ham79});
  \item stable equivalence of operator systems;
\item propagation number as a new invariant under stable equivalence;
\item extreme rays in the cone of positive elements of an operator system;
  \item pure state spaces of operator systems.
\end{itemize}

We will give many examples of the theory, based on spectral truncations of the circle and spaces at finite resolution. This allows to test the above concepts for some concrete operator systems. We will consider:
\begin{itemize}
\item the Toeplitz operator system $\Toep{n}$ arising from spectral truncations of the circle;
\item elements in the group algebra of $\Z$ of finite support: $\CZ{n}$;
\item the circulant matrices, or, equivalently the group algebra $C^*(C_m)$ of the cyclic group of order $m$;
  \item operator systems $E(\cR)$ associated to tolerance relations $\cR$; in particular describing metric spaces with finite resolution.
  \end{itemize}

The paper is organized as follows. In Section \ref{sect:op-syst} we review and develop some concepts and techniques for operator spaces and operator systems, including the appropriate maps between them. This also includes a discussion on the $C^*$-envelope, first introduced by Arveson \cite{Arv69} but realized by Hamana in \cite{Ham79}. We introduce a so-called {\em propagation number} which measures how far an operator system is from the $C^*$-envelope. We show that this is an invariant under stable equivalence of operator systems.

In Section \ref{sect:truncations} we come to our main motivation: spectral truncations. For the circle we present a fully detailed analysis of the structure of the state space for the smallest non-trivial truncation (of rank 3) but which already turns out to be extremely rich.

We continue our analysis in Section \ref{sect:toeplitz} where the underlying mathematical structure of Toeplitz operator systems is unveiled and analyzed in full detail. We identify the $C^*$-envelope and compute the propagation number. The dual operator system is realized in terms of functions in the group algebra of $\Z$ with  support in a fixed interval. Because of the close relation with old factorization results of positive functions on the circle by Fej\'er and Riesz we will call this system the Fej\'er--Riesz operator system. Using the duality we reach a full understanding of the pure state spaces and extreme rays, both for the Toeplitz operator system, as well as for the Fej\'er--Riesz operator system. Moreover, the duality allows for a new proof of another old result by Carath\'eodory on Vandermonde factorizations of positive Toeplitz matrices. One interesting feature which arises in the algebraic geometry of the boundary hypersurface  of the cone of positive Toeplitz matrices is the link between the rank of the matrix and the singularity of the hypersurface (see Theorem \ref{thmstratification}). In \S \ref{sect:trunc-dist} we analyse the distance function on the truncated circle and prove in Theorem \ref{thmtruncatedd} that it is larger than the Kantorovich distance of the corresponding probability measures on the circle. Other results in this direction have been reported in \cite{ALM14,GS19a,GS19b} while further Gromov--Hausdorff convergence results will be reported elsewhere by the second author.

The relation between Toeplitz and circulant matrices is analyzed in Section \ref{sect:circulant}. We realize the finite Fourier transform in terms of a duality between operator systems. 

In the final Section \ref{sect:outlook} we explain how the framework  proposed in this paper \ie using operator systems rather than $C^*$-algebras, allows one to apply the fundamental idea of noncommutative geometry of associating a noncommutative $C^*$-algebra to a quotient space which is intractable by standard topological methods, to situations where the equivalence relation defining the quotient is no longer assumed to be transitive. The detailed development of this idea will be done in a forthcoming paper.

\subsubsection*{Acknowledgements}

We are grateful to Gilles Pisier for his useful remarks on an early version of this paper. We thank Fabien Besnard for very useful comments. 
WvS would like to thank IH\'ES for their hospitality and support during a visit in February 2020. WvS thanks Bram Mesland for numerous fruitful discussions on operator spaces. We thank Hugo Woerdeman for fruitful interaction. 
\section{Preliminaries on operator systems}
\label{sect:op-syst}

In this section we review and develop some of the general concepts and techniques on operator systems that are needed in the later sections.

\subsection{Operator spaces and operator systems}
We start by briefly recalling the theory of operator spaces and operator systems, referring to \cite{ER00,Pau02,Pis03,Bla06} for more details.

\subsubsection{Operator spaces}
Operator space theory can be considered as a ``quantum'' or noncommutative version of Banach space theory in the sense that one extends the usual norms on a vector space $E$ to so-called matrix-norms, {\em i.e.} norms on $M_n(E)$ for every $n \in \N$. Let us make this more precise ({\em cf.} \cite[Section 2.1]{ER00} for more details).
\begin{defn}
  Let $E$ be a vector space. A {\em matrix norm} $\| \cdot \|$ on $E$ is an assignment of a norm $\| \cdot \|_n$ on the matrix space $M_n(E)$ for each $n \in \N$.

  An (abstract) operator space is a linear space $E$ together with a matrix norm $\| \cdot \|$ for which
  \begin{enumerate}
    \item $E$ is complete as a normed vector space.
  \item $\| x \oplus y \|_{m+n} = \max \{ \| x \|_m, \| y \|_n\}$
    \item $ \| \alpha x \beta \|_n \leq \| \alpha \| \| x \|_m \| \beta\|$
  \end{enumerate}
  for all $x \in M_m(E), y \in M_n(E)$ and $ \alpha \in M_{nm}(\C), \beta \in M_{mn}(\C)$. 
  \end{defn}
From the first condition it actually follows that all $M_n(E)$ are complete with respect to the norm $\|\cdot \|_n$ ({\em cf.} \cite[Section 2.1]{ER00}).

Given two operator spaces $E$ and $F$ and a linear mapping $\phi:E \to F$, for each $n \in \N$ there is a corresponding linear map $\phi_n: M_n(E) \to M_n(F)$ from matrices with coefficients in $E$ to matrices with coefficients in $F$, given by
$$
\phi_n(x) = \left( \phi(x_{ij}) \right); \qquad x = (x_{ij}) \in M_n(E).
$$
To each $\phi_n$ we may associate its operator norm and the {\em completely bounded norm} is defined to be
$$
\| \phi\|_\cb := \sup \{ \| \phi_n \|: n \in \N \}.
$$
There are the following notions of morphisms between operator spaces and operator systems. 
\begin{defn}
  Let $\phi:E \to F$ be a linear map between operator spaces.
  \begin{enumerate}
\item We say that $\phi$ is {\em completely bounded} (respectively, {\em completely contractive}) if $\| \phi \|_{\cb} < \infty$ (respectively, $\| \phi \|_{\cb} \leq 1$).
  \item We say that $\phi$ is {\em  completely isometric} if each $\phi_n$ is isometric.
  \end{enumerate}
 \end{defn}

The prototypical example of an operator space is given by a closed subspace of $\B(H)$ for some Hilbert space $H$. Indeed, there is a natural inclusion $M_n(E) \subseteq M_n(\B(H)) = \B(H^n)$ which determines a norm $\| \cdot \|_n$ on $M_n(E)$. We call such an $E$ a {\em concrete operator space}. It follows from Ruan's representation theorem \cite{Rua88} ({\em cf.} \cite[Section 2.3]{ER00} that any abstract operator space is completely isometrically isomorphic to a concrete operator space.


\subsubsection{Operator systems}
We now focus our attention on operator systems, with the crucial property that they possess cones of positive elements. Again there is a notion of abstract operator system and concrete operator system, in fact both originating from the seminal work by Choi and Effros \cite{CE77}. Let us briefly sketch these notions, referring to the original \cite{CE77} and e.g. \cite[Chapter 13]{Pau02} for more details.

Let $E$ be a vector space equipped with a conjugate linear involution $x \mapsto x^*$. We call such a space a $*$-vector space and we set $E_h = \{ x \in E: x^* = x\}$. For $(x_{ij}) \in M_n(E)$ we set $(x_{ij})^* = (x_{ji}^*)$ so that $M_n(E)$ is also a $*$-vector space. In order to talk about positive elements we need a notion of ordering.

\begin{defn}\label{defn:matrix-order}
We say that a $*$-vector space is {\em matrix ordered} if
\begin{enumerate}
\item for each $n$ we are given a cone of positive elements $M_n(E)_+$ in $M_n(E)_h$,
\item $M_n(E)_+ \cap (-M_n(E)_+) = \{ 0\}$ for all $n$,
  \item for every $m,n$ and $A  \in M_{mn}(\C)$ we have that $A M_n(E)_+ A^* \subseteq M_m(E)_+$.
\end{enumerate}
\end{defn}
We will write $x \geq 0$ and call $x$ positive whenever $x  \in M_n(E)_+$. A map from $M_n(E)$ to $M_n(F)$ is then called {\em positive} if it maps $M_n(E)_+$ to $M_n(F)_+$. 
\begin{defn}
  Let $\phi:E \to F$ be a linear map between matrix-ordered $*$-vector spaces. 
    \begin{enumerate}
\item We call $\phi$ {\em completely positive} if each $\phi_n$ is positive.
\item We call $\phi$ a {\em complete order isomorphism} if $\phi$ is invertible with both $\phi$ and $\phi^{-1}$ completely positive.
  \end{enumerate}
  \end{defn}
We also say that $\phi:E \to F$ is a {\em complete order injection} if it is a complete order isomorphism onto its image.

Finally, let us address the role that $1$ plays in an operator system. Let $E$ be an ordered $*$-vector space. We call $e \in E_h$ an {\em order unit} for $E$ if for each $x \in E_h$ there is a $t>0$ such that $-t e \leq x \leq t e$. It is called an {\em Archimedean order unit} if $-t e \leq x$ for all $t >0$ implies that $x \geq 0$.
\begin{defn}
  An {\em (abstract) operator system} is given by a matrix-ordered $*$-vector space $E$ with an order unit $e$ such that for all $n$
  $$I_n =\begin{pmatrix} e &0 &  \cdots & 0 \\ 0 & \ddots & \ddots & \vdots \\ \vdots &\ddots & \ddots & 0 \\ 0 & \cdots & 0 & e \end{pmatrix}
  $$ is an Archimedean order unit for $M_n(E)$.   
  \end{defn}

There is a relation between operator systems and operator spaces and, in fact, a matrix order induces a matrix norm ({\em cf.} \cite[Proposition 13.3]{Pau02}):
$$
\| x \|_n = \inf \left\{ t: \begin{pmatrix} t I_n & x \\ x^* & t I_n \end{pmatrix} \geq 0 \right\} 
$$
for any $x \in M_n(E)$. This relationship respects the morphisms between operator systems and operator spaces, as the following result shows.

\begin{prop}
\label{prop:rel-maps-opsyst}

  Let $\phi:E \to F$ be a linear map between operator systems.
  \begin{enumerate}
  \item If $\phi$ is completely positive, then it is completely bounded with
    $$
\| \phi\|_{\cb} = \| \phi\| = \|\phi(1) \|.
$$
  \item A unital map $\phi$ is completely positive if and only it is completely contractive.

  \item A unital map $\phi$ is a complete order injection if and only if it is completely isometric.
  \end{enumerate}
 
\end{prop}
\proof
The first two statements can be found in \cite{ER00} as Lemma 5.1.1 and Lemma 5.1.2, respectively, the third follows then directly from the second ({\em cf.} \cite[\S 1.3.3]{BM04}).
\endproof
We may summarize this by saying that there is a functor from abstract operator systems to abstract operator spaces. The obtained abstract operator spaces have naturally a unit and involution, and as discussed below in the concrete case, this additional structure is the only nuance between the two notions. In fact the role of the unit in passing from operator spaces to operator systems was
fully clarified by the work of D. Blecher and M. Neal \cite{BN11} who found the norm identities that qualify an element of an operator space as a unit of an operator system with the given underlying operator space.

\begin{defn}  We say that a subspace $E \subseteq \B(H)$ is a (concrete) {\em operator system} if it is self-adjoint in the sense that $E^* = E$ where $E^* = \{ x: x^* \in E\}$ and contains the identity $1$ in $\B(H)$.
\end{defn}

The {\em cone of positive elements} in $E$ is defined to be
$$
E_+ := E \cap \B(H)_+,
$$
and, more generally, we write for any $n \in \N$:
$$
M_n(E)_+ := M_n(E) \cap \B(H^n)_+.
$$
This turns a concrete operator system in an abstract operator system. 
In the other direction, the celebrated Choi--Effros Theorem shows that any abstract operator system is completely order isomorphic to a concrete operator system \cite{CE77}.

We also note that a unital complete order isomorphism $\phi:A \to B$ between two unital $C^*$-algebras is a $*$-isomorphism (see \cite[Corollary 5.2.3]{ER00} for a proof).

We also have the following result.
\begin{prop}\label{prop;iso-unital}
Let $\phi: E \to F$ be a completely isometric, completely positive isomorphism between unital operator systems. Then $\phi$ is unital.
  \end{prop}
\proof
This follows since the unit of a unital  operator system is characterized uniquely as the largest element among positive elements of norm $\leq 1$.
\endproof

\subsubsection{States on operator systems}
\label{sect:states-opsyst}
One of the advantages of working with operator systems is that there is a notion of {\em states}, defined as positive linear functionals of norm $1$. Since for linear functionals $\phi:E \to \C$ we have $\| \phi \|_\cb = \| \phi \|$ ({\em cf.} \cite[Corollary 2.2.3]{ER00} or \cite[Proposition 3.8]{Pau02}), Proposition \ref{prop:rel-maps-opsyst} above implies that a state on $E$ can equivalently be defined to be a linear functional $\phi: E \to \C$ such that 
$$
 \| \phi \| =\phi(1) = 1.
 $$
 This is completely analogous to the case of $C^*$-algebras (see for instance \cite[Proposition II.6.2.5]{Bla06}) and, in fact, for states we automatically have complete positivity (\cite[Proposition 3.8]{Pau02}).

 In any case, we may talk about the {\em state space} $\cS(E)$ of the operator system $E$; it is a convex space which is compact for the weak $*$-topology, so that Choquet theory applies \cite{Phe01}. The {\em pure states} are then given by the extreme points in $\cS(E)$. We call the weak $*$-closure of the set of extreme points in $\cS(E)$ the {\em pure state space}; it will be denoted by $\cP(E)$.

 We also record from \cite[Section II.6.3]{Bla06} that if $E \subseteq A$ with $A$ a $C^*$-algebra, and $\phi$ is a state on $E$, then by the Hahn--Banach Theorem $\phi$ extends to a functional $\psi$ on $A$ of norm one. Since $\psi(1)=1$, $\psi$ is a state on $A$. For pure states on $E$ there is the following well-known result ({\em cf.} \cite[Proposition 2.3.24]{BR87}):. 
\begin{fact}
Let $\phi: E \to \C$ be a pure state. The set $V_\phi$ of extensions of $\phi$ to states on $A$ is a compact convex subset of $\cS(A)$ (in the weak-$*$-topology) and any extreme point of this set is a pure state on $A$. In particular, $\phi$ allows for an extension to a pure state on $A$.
  \end{fact} 

\subsection{Non-unital operator systems}

For non-unital operator systems we shall use the results of Werner \cite{Wer02}
together with the following correction needed since it is wrongly stated in that
paper that the state space $\cS_n(E)$ of any operator space equipped with a matrix order is compact. This fails even for $C^*$-algebras 
such as the algebra of sequences tending to $0$ at $\infty$. However the problem
is fixed using the following fact:
\begin{lma}
\label{lma:states-compactness}
Let $\cS_n(E)$ be the state space consisting of positive linear functionals on $M_n(E)$ of norm 1. The rescaled state space $\tilde \cS_n(E):=\{\lambda \phi\mid \phi \in  \cS_n, \lambda\in [0,1]\}$
is weakly compact. Any continuous functional homogeneous of degree 1 on $\tilde S_n(E)$ reaches its maximum
on $\cS_n(E)\subseteq \tilde \cS_n(E)$.
\end{lma}

Werner considers {\em matrix-ordered operator spaces}. These are defined to be operator spaces $E$ with a matrix order as in Definition \ref{defn:matrix-order} with the additional properties that
\begin{enumerate}
\item the cones $M_n(E)_+$ are all closed, and,
\item the involution is an isometry on $M_n(E)$.
\end{enumerate}
He then constructs ``partial unitizations''  for arbitrary matrix-ordered operator spaces after proving their uniqueness (\opcit Lemma 4.3). His construction proceeds as follows:
\begin{defn}
\label{defn:unit-opsyst}
Let $E$ be a matrix-ordered operator space and define $A_\eps = A + \eps \id_n$ for every matrix $A \in M_n(\C)$. On the space $E \oplus \C$ we define
\begin{enumerate}
\item $(x,A)^* = (x^*,A^*)$ for all $(x,A) \in M_n(E )\oplus M_n( \C)$,
\item for any $(x,A) \in M_n(E \oplus \C)_h$ we set
$$ (x,A) \geq 0 \qquad \text{iff } \quad A \geq 0 \text{ and } \phi(A_\eps^{-1/2} x A_\eps^{-1/2}) \geq -1
$$
for all $\eps>0$ and $\phi \in \cS_n(E)$.
\end{enumerate}
We denote by $E^\sharp$ the space $E \oplus \C$ equipped with this order structure. 
\end{defn}
When $E$ is the matrix-ordered operator space  associated to a possibly non-unital $C^*$-algebra $B$ this construction agrees with the traditional  adjunction of a unit $B\subset B^\sharp$ (see \opcit Corollary 4.17).
As shown in \opcit Lemmas 4.8 and 4.9, one has
\begin{prop}
\label{prop:unit-opsyst}
$(i)$~Let $E$ be a matrix-ordered operator space. The space $E^\sharp$ defined above is a (unital) operator system. \newline
$(ii)$~Let $T:E \to F$ be a completely contractive and completely positive map between matrix-ordered operator spaces $E,F$. Then the natural unital extension $T^\sharp: E^\sharp \to F^\sharp$ is completely positive.
\end{prop}
For $x\in E$ one lets $\nu^0_E(x):= \sup \{ |\phi(x)|: \phi \in \cS(E)\}$.
The number $\nu^0_E(x)$ is the so-called {\em numerical radius} of $x \in E$; it also makes sense more generally for $x \in M_n(E)$. Moreover, it allows to introduce a norm on $M_n(E)$ as follows (\cite[Lemma 3.1]{Wer02}). 
\begin{lma}
\label{lma:num-radius}
Let $\nu^0_E$ be the numerical radius. 
\newline
$(i)$~The map $\nu_E: x \mapsto \nu_E^0 \left(\begin{smallmatrix} 0 & x \\ x^* & 0 \end{smallmatrix}\right)$ defines a norm on $M_n(E)$ and we have $\nu_E(\cdot) \leq \| \cdot \|$. \newline 
$(ii)$~The inclusion map $\imath_E: E \to E^\sharp$ is completely contractive and completely positive.\newline
$(iii)$~The inclusion $\imath_E$ is a complete isometry when considered as a map from $(E,\nu_E)$ to $E^\sharp$.
\end{lma}
 In fact, in \cite[Lemma 4.5]{Wer02} it is shown that if the embedding $E \to E^\sharp$ is completely positive, then it is completely contractive if and only if it is a complete isometry between $(E,\nu_E)$ and $E^\sharp$. It is here that Lemma \ref{lma:states-compactness} should be used to conclude (in line 7 of the proof of \cite[Lemma 4.5]{Wer02}) that the supremum $\nu_E^0$ is actually attained by a $\phi \in \cS_n(E)$.

\begin{corl}
Let $T:E \to F$ be a completely isometric, complete order isomorphism. Then $T^\sharp: E^\sharp \to F^\sharp$ is a (unital) complete order isomorphism. 
\end{corl}

 In line with the result in \cite{Wer02} we will make the following definition.
 \begin{defn}
A {\em non-unital operator system} is given by a matrix-ordered operator space for which the norm $\nu_E(\cdot)$ coincides with the norm $\| \cdot \|$. 
  
   \end{defn}
For a non-unital operator system the inclusion $\imath_E : E \to E^\sharp$ is a complete isometry, and $E^\sharp$ will be called a {\em unitization} of $E$. 
 
\medskip

 The main result of \cite{Wer02} is then an analogue of the Choi--Effros Theorem for non-unital operator systems. Indeed, since $E^\sharp$ is an (abstract) unital operator system it can be realized as a concrete operator system in $\B(H)$ for some Hilbert space $H$. It then follows that if $\nu_E(\cdot)$ and $\|\cdot \|$ coincide on $M_n(E)$, then we can also realize $E$ via the (completely isometric) inclusion map $E \to E^\sharp$ as a (concrete) non-unital operator system in $\B(H)$ (see \cite[Corollary 4.11]{Wer02}).

Note that in the definition of a matrix-ordered operator space there is no requirement of non-triviality of the positive cone. In particular starting with an operator space $E$ (with isometric involution) one can consider the trivial matrix order $M_n(E)_+=\{0\}$ for all $n$. The norm $\nu_E$ is then the same as the original norm on $E$ since positivity of functionals is automatic. But then the Choi--Effros Theorem applied to the partial unitization $E^\sharp$ implies Ruan's result for operator spaces.

In all the examples of matrix-ordered operator spaces considered in this paper, the following non-triviality condition holds: the cones $M_n(E)_+$ span $M_n(E)$, or more precisely (and for all $n$)
\begin{equation}\label{spanning}
M_n(E)_+-M_n(E)_+=\{x\in M_n(E)\mid x=x^*\}.
\end{equation}
To understand the meaning of this condition in terms of the partial unitization $E^\sharp$ note the following 
\begin{fact} Let $A$ be a unital $C^*$-algebra and $\phi$ a pure state on $A$. Then, with the notations $\ker(\phi)_+:=\ker(\phi)\cap A_+$ and $\ker(\phi)_{sa}:=\ker(\phi)\cap A_{sa}$ one has
$$
\phi(xy)=\phi(x)\phi(y)\qqq x,y\in A\iff \ker(\phi)_+-\ker(\phi)_+=\ker(\phi)_{sa}.
$$	
\end{fact}
\proof If $\phi$ is a morphism $\phi:A\to \C$ and $x=x^*\in A$ fulfills $\phi(x)=0$ then with $\vert x\vert =\sqrt{x^*x}$ one has $\phi(\vert x\vert)=0$ and $x=\vert x\vert-y$ where also $y\in \ker(\phi)_+$. Conversely if the irreducible GNS representation $(H_\phi,\pi_\phi,\xi_\phi)$ is of dimension $>1$ one finds a self-adjoint element $a\in A$ such that $\pi_\phi(a)\xi_\phi\neq 0$ but $\pi_\phi(a)\xi_\phi\perp \xi_\phi$. Then one has $\phi(a)=0$ but $a\notin \ker(\phi)_+-\ker(\phi)_+$
since elements $b\in \ker(\phi)_+$ all fulfill  $\pi_\phi(b)\xi_\phi=0$ since $\Vert\pi_\phi(b^{1/2})\xi_\phi\Vert^2=\phi(b)=0 \Rightarrow \pi_\phi(b)\xi_\phi=0$.\endproof 
Thus one can define in general a {\em character} of a unital operator system as a pure state $\phi$ such that $\ker(\phi)_+-\ker(\phi)_+=\ker(\phi)_{sa}$. Then condition \eqref{spanning} on a matrix-ordered operator space means that the canonical state on the partial unitization $E^\sharp$ is a character.

\subsection{Duals of operator systems}
Already in \cite{CE77} Choi and Effros analyzed the notion of duality for operator systems, which we now briefly discuss here. See also \cite{FP12} for a more recent perspective. In general, duals of operator systems are only matrix-order vector spaces, but in the finite-dimensional case also an Archimedean order unit can be constructed. Since our main interest in this type of duality is for finite-dimensional Toeplitz matrices ({\em cf.} Section \ref{sect:duality-toeplitz} below) we will here restrict to this case.

So let $E$ be a finite-dimensional (abstract) operator system $E$. We let $E^d$ be the dual vector space of $E$ and let $M_n(E^d)$ be paired component-wise with $M_n(E)$:
$$
\phi(x) = (\phi_{ij}) (x_{ij}) = \sum_{ij} \phi_{ij}(x_{ij})
$$
where $\phi= (\phi_{ij}) \in M_n(E^d)$ and $x= (x_{ij}) \in M_n(E)$. We define a matrix order on $E^d$ by
$$
M_n(E^d)_+ = \left \{ \phi \in M_n(E^d): \phi(x) \geq 0 \text{ for all } x \in M_n(E)_+ \right\}.
$$
One quickly checks that this is a matrix-order, since for any $A \in M_{mn}(\C)$ we have 
$$
(A \phi A^*) (x) = \phi (A^t x (A^t)^*)
$$
so that $A \phi A^* \in M_n(E^d)_+$ if $\phi$ is, because $M_m(E)_+$ is closed under conjugation by a scalar-valued matrix.

Let us now consider the existence of an Archimedean order unit. The notion of faithful state makes sense for any operator system: a state $\phi$ is {\em faithful} if $\phi(x)>0$ for $x>0$.

\begin{prop}[Choi--Effros] let $E$ be a finite-dimensional (abstract) operator system $E$.  Let $E^d$ be the dual vector space of $E$, equipped with the above matrix-ordering. A  state $\chi$ on $E$ defines an order unit on $E^d$ if and only if it is faithful. Then $\text{diag} (\chi, \ldots, \chi)$ is an Archimedean order unit on $M_n(E^d)$.  Faithful states exist and endow  $E^d$ with the structure of an operator system. 
\end{prop}
\proof The result follows from the existence of a compact base $K$ for $E_+$. For any $x\in K$ there exists $\phi \in E^d_+$ with $\phi(x)>0$ thus, by compactness there exists a faithful state $\chi: E \to \C$. It is an  order unit for $E^d$ since the compact set $\chi(K)\subset (0,\infty)$ is bounded away from $0$ while $\phi(K)$ is bounded for any $\phi \in E^d_h$.
It is also Archimedean since if $\phi+t \chi\in E^d_+$ for all $t>0$ one, has for any $x \in E_+$, that $\phi(x)+t \chi(x)\geq 0 $ for all $t>0$ and thus  $\phi(x)\geq 0$. The  extension to $M_n(E^d)$ is straightforward. 
\endproof

\begin{lma}
  \label{lma:dual} Let $E$ be a finite-dimensional (abstract) operator system $E$.
There is a complete order isomorphism of operator systems $(E^d)^d \cong E$. 
  \end{lma}
\proof
This is a straightforward application of the bipolar theorem, stating in the finite-dimensional case that $((E^d)^d)_+ = E_+$.
\endproof

Let us now consider maps between operator systems and their duals. Clearly, if $\phi:E \to F$ there is the induced map $\phi^d : F^d \to E^d$ given by
$$
\phi^d(f) (x) = f( \phi(x)); \qquad (f \in F^d, x \in E).
$$
\begin{prop}
If $E$ and $F$ are operator systems. A linear map $\phi:E \to F$ is completely positive if and only if $\phi^d: F^d \to E^d$ is completely positive. 
  \end{prop}
\proof
In view of Lemma \ref{lma:dual} it is sufficient to prove one implication. So suppose that $\phi$ is completely positive, {\em i.e.} $\phi_{(n)} \geq 0$ for all $ n \geq 0$. Then $\phi_{(n)}: M_n(F^d) \to M_n(E^d)$ satisfies
$$
\phi_{(n)}^d(\psi)(x) = \phi_{(n)}^d (\psi_{ij} ) (x_{ij}) = \sum_{ij} \phi^d(f_{ij}) (x_{ij}) = \sum_{ij} f_{ij}(\phi(x_{ij}))= \psi(\phi_{(n)}(x))
$$
for $\psi= (\psi_{ij}) \in M_n(F^d), x=(x_{ij}) \in M_n(E)$. Hence if $\psi \geq 0$ it follows that $\phi_{(n)}^d(\psi) \geq 0$.
\endproof

\begin{corl}
\label{corl:dual-extr}
Extreme rays in the cone $E_+$ are in one-to-one correspondence to the pure states of $E^d$ and, vice versa, pure states of $E$ are in one-to-one correspondence to extreme rays in the cone $(E^d)_+$.
  \end{corl}

\subsection{$C^*$-envelopes of operator systems}
In \cite{Arv69} Arveson introduced the notion of a $C^*$-envelope of an operator system. Their existence and uniqueness was established in full generality by Hamana in \cite{Ham79} based on the theory of injective envelopes (see also \cite[Ch. 15]{Pau02} and \cite[Section 4.3]{BM04}). More recently, in \cite{Arv03,Arv08} Arveson revisited his original approach (using so-called boundary representations) to the problem of $C^*$-envelopes, basing himself on the work of Dritschel and McCullough \cite{DM05}. In this context, we also mention the paper by Arveson's student and grand-student  \cite{DK15}. We here briefly recall some of these notions and the main result. We allow for non-unital operator systems. 

\begin{defn}\label{defnextension} 
Let $E$ be an operator system. A {\em \sext-extension} $\kappa: E \to A$ of $E$ is given by a completely isometric and completely positive map such that $A = C^*(\kappa(E))$ and  $\kappa^\sharp: E^\sharp \to A^\sharp$
 is a complete order isomorphism onto its range. 
\end{defn}
 The above definition contains the usual one of a $C^*$-extension in the case of unital operator systems ({\em i.e.} a unital order isomorphism onto its range, {\em cf.} \cite[Definition 2.1]{Ham79} as we show now:
 \begin{lem}\label{lemm-1}
Let  $E$ be a unital operator system. Then any $C^*$-extension of $E$ is a \sext-extension.
 \end{lem}
\proof The proof is based on the fact that states on unital operator systems, in the sense of positive linear functionals of norm 1, are automatically unital. Thus given a $C^*$-extension $\kappa: E \to A$ states on $E$ are restrictions of states on $A$ and one obtains that $\kappa^\sharp: E^\sharp \to A^\sharp$
 is a complete order isomorphism onto its range. \endproof 
In fact in the special case of unital operator systems the notion of \sext-extension is more general than the usual notion, as shown by the following:
\begin{ex}\label{sim-ex} Let $E$ be the smallest unital operator system consisting of scalar multiples of the unit $1_E$. Let $A:=C_0([0,\infty))$ be the $C^*$-algebra of continuous functions vanishing at $\infty$ on $[0,\infty)$. Let $\kappa :E \to A$ be given  by $\kappa(1_E)=h$ with $h(x):=\exp(-x)$. By construction $\kappa$ is a completely isometric and completely positive map and its range generates $A$ as a $C^*$-algebra. Moreover the map $\kappa^\sharp: E^\sharp \to A^\sharp$ is a complete order isomorphism onto its range since evaluation at $0\in [0,\infty)$ gives a completely positive retraction $\sigma:A\to E$ of the map $\kappa$.
	
\end{ex}

Let $\phi:E \to F$ be a completely isometric, complete order isomorphism of operator systems. We will say that two \sext-extensions $\kappa :E \to A$ and $\lambda :F \to B$ are {\em equivalent} if there is a $*$-isomorphism $\rho: A \to B$ such that $\rho \circ \kappa = \lambda \circ \phi$. 

\begin{defn}\label{defn-sext-envelope}
  Let $E$ be an operator system. A {\em \sext-envelope} is a \sext-extension $\kappa: E \to A$ with the following universal property: for every \sext-extension $(B,\lambda)$ there exists a unique surjective $*$-homomorphism $\rho: B \to A$ such that $\rho \circ \lambda = \kappa$. 
\end{defn}
Existence of the $C^*$-envelope for unital operator systems was shown by Hamana in \cite{Ham79} and we refer to that paper and \cite[Ch. 15]{Pau02} and \cite[Section 4.3]{BM04} for the proof. We now deal with the non-unital case.

\begin{thm}\label{sext-envelope}
 $(i)$~The \sext-envelope of a non necessarily unital operator system $E$ exists and is unique (up to equivalence).\newline
 $(ii)$~If the system is unital the \sext-envelope is equal to the $C^*$-envelope.
  \end{thm}
\proof 
$(i)$~Let us show existence of a \sext-envelope in the case that $E \subseteq \B(H)$ is non-unital. Let $ E^\sharp = E \oplus \C$ be the  unitization of $E$ as defined in Definition \ref{defn:unit-opsyst}; it is a unital operator system and so it has a $C^*$-envelope; let us denote this by $\kappa: E^\sharp \to B$.  We claim that $A=C^*(\kappa(E)) \subset B$ is a \sext-envelope of $E$. First by  Corollary 4.17 of \cite{Wer02} the unitization $C^*(\kappa(E))^\sharp$ is the usual $C^*$-algebra unitization. The map $\imath_E :E \to E^\sharp$ is completely isometric and completely positive and so is $\alpha=\kappa\circ\imath_E :E \to A$. To show that $\alpha$ is  a \sext-extension one needs to prove that $\alpha^\sharp:E^\sharp\to A^\sharp$ is a complete order isomorphism on its range.  Since $B$ is a unital $C^*$-algebra and $A\subset B$ a $C^*$-subalgebra one has a canonical morphism $\beta : A^\sharp\to B$. It extends the inclusion  by sending the adjoined unit of $A^\sharp$ to $1_B$. Moreover  
$$
\beta \circ \alpha^\sharp= \kappa 
$$
If an element $x\in \alpha^\sharp(E^\sharp)$,  is positive in $A^\sharp$ then $\beta(x)\in B$ is positive and since $\kappa: E^\sharp \to B$ is an order isomorphism on its range there exists uniquely a positive element $y\in E^\sharp$ with $\kappa(y)=\beta(x)$. Let $z\in E^\sharp$ with $\alpha^\sharp(z)=x$, one has $$\kappa(z)=\beta \circ \alpha^\sharp(z)=\beta(x)=\kappa(y)$$ and hence $z=y$ since $\kappa$ is an injection, so that $z$ is positive. The same argument applies to matrices and shows that $\alpha$ is a \sext-extension. Now suppose that $\lambda: E \to C$ is some other \sext-extension. Then $C^\sharp$ is  the $C^*$-algebra unitization of $C$ and $\lambda^\sharp:E^\sharp \to C^\sharp$ is a (unital) $C^*$-extension by Definition \ref{defnextension}. Hence there exists a surjective $*$-homomorphism $\rho: C^\sharp \to B$, with $B$ as defined above, such that $\rho\circ \lambda^\sharp = \kappa$. The $*$-homomorphism given by the restriction $\rho|_C: C \to B$ lands in $A=C^*(\kappa(E))$ since $C = C^*(\lambda(E))$ by Definition \ref{defnextension}. Thus we find that $\rho'=\rho|_C : C \to A$ is a surjection, and that $\rho'\circ \lambda = \kappa\circ\imath_E= \alpha$  as desired.

For uniqueness, assume that $\kappa: E \to A$ and $\lambda:E \to B$ are two \sext-envelopes of $E$. The universal property of both give two surjective $*$-homomorphisms $\sigma :A \to B$ and $\rho: B \to A$ such that $\sigma \circ \kappa = \lambda$ and $\rho \circ \lambda =\kappa$. As a consequence $\rho \circ \sigma \circ \kappa = \kappa$, that is to say, $\rho \circ \sigma$ is the identity when restricted to $\kappa(E) \subset A$. But since $\rho$ and $\sigma$ are $*$-homomorphisms and $A$ is generated by $\kappa(E)$ it follows that $\rho \circ \sigma = \text{id}_A$. Similarly, we find $\sigma \circ \rho= \text{id}_B$ so that $A \cong B$, compatibly with the extension maps $\kappa$ and $\lambda$.\newline
$(ii)$~Let  $E$ be a unital operator system, and $\imath_E: E \to C^*_{\rm env}(E)$ its $C^*$-envelope. Let $\kappa:E\to A$ be the \sext-envelope of $E$. Then by Lemma \ref{lemm-1}, $\imath_E$ is a \sext-extension. Thus by the universal property of the \sext-envelope, there exists a unique surjective $*$-homomorphism $\rho: C^*_{\rm env}(E) \to A$ such that $\rho \circ \imath_E = \kappa$. Since the $C^*$-algebra $C^*_{\rm env}(E)$ is unital it follows that $A$ is unital, with unit $1_A$ equal to $\rho(1)$. Let $1_E$ be the unit of the operator system $E$. It follows that 
$$
\kappa(1_E)=\rho \circ \imath_E(1_E)=\rho(1)=1_A.
$$ 
This shows that $\kappa:E\to A$ is a $C^*$-envelope. Thus by the universal property of $C^*$-envelopes there exists uniquely a surjective  $*$-homomorphism $\rho': A\to C^*_{\rm env}(E)$ such that $\rho' \circ \kappa= \imath_E$. One then concludes that $\rho'$ is the inverse of $\rho$ since both $C^*_{\rm env}(E)$ and $A$ are generated by the image of the operator system $E$.
\endproof

\begin{corl} \label{corext} Let  $E$ be a unital operator system. Then
$C^*_{\rm env}(E)$ is the $C^*$-algebra generated by $E$ in $C^*_{\rm env}(E^\sharp)$.
\end{corl}
\proof By construction the $C^*$-algebra generated by $E$ in $C^*_{\rm env}(E^\sharp)$ is the \sext-envelope of $E$ and by $(ii)$ of Theorem \ref{sext-envelope} it coincides with the $C^*$-envelope: $\imath_E: E \to C^*_{\rm env}(E)$.\endproof

\begin{rem}\label{rem:example} It is important in definition \ref{defnextension} to assume that the associated map $\kappa^\sharp$ is an order isomorphism with its range. The following example shows that if one drops this hypothesis the \sext-envelope no longer exists. Consider the non-unital system $S$ formed of a single self-adjoint $H$ with $\Vert H\Vert=1$, and where
the positive cone is $\{0\}$. Then if one weakens definition \ref{defnextension} by dropping the requirement on $\kappa^\sharp$, a $C^*$-extension $\kappa: E \to A$ is simply a self-adjoint generating element  $h\in A$ of norm $\Vert h\Vert = 1$. In particular one can have $h>0$ and one sees that this rules out the existence of a $C^*$-envelope since $h$ does not contain $-1$ in its spectrum. But the system $S^\sharp$ does have a $C^*$-envelope which is the $C^*$-algebra $C(\{\pm 1\})$ and where $H$ is the function $H(\pm 1):=\pm 1$.
	
\end{rem}

Part $(ii)$ of Theorem \ref{sext-envelope} shows that we can drop the distinction between $C^*$-envelope and \sext-envelope. In the following, we will use the terminology $C^*$-envelope of $E$ and denote it as $\imath_E: E \to C^*_{\rm env}(E)$. 

\subsection{ \v{S}ilov boundary ideals}
There is a useful description of the $C^*$-envelope in terms of \v{S}ilov boundary ideals that we now recall \cite{Arv69,Ham79}. We shall only use it in the unital case and restrict to this case in this subsection.

\begin{defn}
  Let $E$ be a unital operator system and $\kappa: E \to A$ a $C^*$-extension. A {\em boundary ideal} for the extension is a closed two-sided ideal $I \subseteq A$ such that the quotient map $q: A \to A/I$ is completely isometric on $\kappa(E) \subseteq A$.

A boundary ideal is called the {\em \v{S}ilov ideal} if it contains every other boundary ideal. 
  \end{defn}



Before we analyze the relation between the \v{S}ilov boundary ideal and the $C^*$-envelope of operator systems, let us spend a few words on the topological origin of these boundary ideals \cite[Sect. 2.1]{Arv69} ({\em cf.} \cite[Chapter 6]{Phe01}). Namely, let $X$ be a compact Hausdorff topological space and consider a linear subspace $E \subseteq C(X)$ that contains the constants and separates points of $X$. Then there is a smallest closed subset $K \subseteq X$ such that every function in $E$ achieves its maximal absolute value on $K$. This is called the {\em \v{S}ilov boundary} of $X$ relative to $E$. In terms of the corresponding ideals we then find that
$$
J = \{ f \in C(X): f(K) = 0\} 
$$
and that the quotient norm in $C(X)/J$ is
$$
\| f|_K \| = \sup_{x \in K} |f(x)|.
$$
But then, to say that $f$ attains its maximum on $K \subseteq X$ amounts to saying that $\|f|_K\| = \|f\|$. Thus, given the one-to-one correspondence between closed subsets in $X$ and closed ideals in $C(X)$ we find that $J$ is a \v{S}ilov ideal for an $E \subseteq C(X)$ if and only if $K$ is \v{S}ilov boundary for $E$.

\begin{ex}
  \label{ex:disc-alg}
  The traditional example of the \v{S}ilov boundary is given by the continuous harmonic functions $C_{\text{harm}}(\bar {\mathbb D})$ on the closed disc. 
  Then by the maximum modulus principle any harmonic function attains its maximum at the boundary of $\mathbb D$. The \v{S}ilov boundary for this operator system is thus given by the circle and the \v{S}ilov boundary ideal is $C_0(\mathbb D)$. 
  \end{ex}

We now return to the description of the $C^*$-envelopes using \v{S}ilov boundary ideals. Note that the following result is nothing but a reformulation of the results in \cite{Ham79}, very much in line with \cite{Arv69}. 

\begin{prop}
\label{prop:silov}
Let $E$ be a unital operator system and let $ \kappa: E \to A$ be a $C^*$-extension. Then there exists a (necessarily unique) \v{S}ilov boundary ideal $J$. Moreover, the $C^*$-envelope $C^*_\env(E)$ is $*$-isomorphic to $A/J$. 
  \end{prop}
\begin{proof}
From the universal property of the $C^*$-envelope $\imath_E :E \to C^*_\env(E)$ it follows that there is a surjective $*$-homomorphism $\pi: A \to C^*_\env(E)$ such that $\pi \circ \kappa = \imath_E$. Hence there is a $*$-isomorphism $\tilde \pi : A /\ker \, \pi \to C_\env^*(E)$ such that
$$ 
\xymatrix@C+2pc{
E \ar@{=}[d] \ar^{\imath_E}[r]  &C^*_\env(E)\\
E \ar[r]^{q \circ \kappa}  & {A/\ker \, \pi} \ar[u]_{\tilde \pi}
}
$$
is a commutative diagram, where $q :A \to A/\ker \, \pi$ denotes the quotient map. We claim that $J= \ker \, \pi$ is the \v{S}ilov boundary ideal.

Indeed, $J$ is a boundary ideal since the restriction of $q: A  \to A/\ker \, \pi$ to $\lambda(E)$ is $q \circ \kappa = \tilde \pi^{-1} \circ \imath_E$ which is surely a complete order isomorphism onto its range.

Next, let $I \subseteq A$ be any boundary ideal with $q' : A \to A/I$ the corresponding quotient map. Then $q' \circ \kappa :E \to A/I$ is a $C^*$-extension of $E$ and thus, by the universal property of $C^*_\env(E) \cong A/J$ there is a surjective $*$-homomorphism
$$
\rho:A /I \to A/J
$$
such that $\rho \circ q' \circ \kappa = q\circ \kappa$. This means that $\rho(x+I) = x+J$ for all $x \in \kappa(E) \subseteq A$. Since $\rho$ is a $*$-homomorphism and $A$ is generated by $\kappa(E)$ it follows that $\rho(x+I) = x+J$ for all $x \in A$. In particular, for each $x \in I$ we have
$$
0+J = \rho(0+I) = \rho(x+I) = x+J
$$
so that $I \subseteq J$. 
\end{proof}

\begin{corl}
  \label{corl:simple-env}
Let $E$ be a unital operator system and let $\kappa :E \to A$ be a $C^*$-extension. If $A$ is a simple $C^*$-algebra then $C_\env^*(E)$ is isomorphic to $A$.
  \end{corl}
\proof
Since $A$ is simple there are no two-sided ideals and in particular the \v{S}ilov boundary ideal is trivial. In other words, the $C^*$-envelope of $E$ is given by $A$. 
\endproof

\begin{ex}\label{ex:disc-alg1}
  Returning to the harmonic functions on the closed disc of Example \ref{ex:disc-alg} we see that $C^*_\env(C_{\text{harm}}(\bar {\mathbb D})) \cong C(\overline{\mathbb{D}}) /C_0(\mathbb D) \cong C(S^1)$.   
  \end{ex}

\begin{ex}
  Let $\theta$ be an irrational real number and let $U$ and $V$ be two unitary operators in a Hilbert space such that
  $$
VU = e^{ 2 \pi i \theta} VU.
$$
The linear span $E$ of the operators $U$ and $V$ together with their adjoints and the identity is a operator system. It is a subspace of the $C^*$-algebra generated by $U$ and $V$, which is of course nothing but the noncommutative torus $A_\theta$. Since the latter is simple Corollary \ref{corl:simple-env} applies and we conclude that the $C^*$-envelope of $E$ is given by the noncommutative torus $A_\theta$.
\end{ex}

\begin{rem}
  We urge the reader to transpose the above results to the non-unital case by replacing $C^*$-extensions by $C^\sharp$-extensions.
  \end{rem}

\subsection{Stable equivalence for operator systems}
Before we can introduce the notion of stable equivalence, we will need to briefly digress on tensor products of operator systems. We will base our approach on \cite{KPTT11} which is focusing completely on operator systems. In fact, the authors develop tensor products from the point of view of abstract operator systems (that is to say, for matrix-ordered order unit spaces). However, the link to concrete operator systems such as $E \subseteq \B(H)$ is also worked out ({\em cf.} \cite[Theorem 4.4]{KPTT11}), and allows us to here make the following `hands-on' definition. For the development of tensor products in the case of non-unital (abstract) operator systems, we refer to \cite{LiNg12}
\begin{defn}
Let $E \subset \B(H)$ and $F \subseteq \B(K)$ be operator systems. We define the {\em minimal tensor product} $E \otimes_\min F$ of $E$ and $F$ as the following norm closure
$$
E \otimes_\min F = \overline{E \otimes F} \subseteq \B( H \otimes K),
$$
where the algebraic tensor product $E \otimes F$ is naturally embedded in $\B(H \otimes K)$.
\end{defn}
The construction of the minimal tensor product only depends on the abstract operator system structure, and it defines a bi-functor from operator systems to operator systems \cite[Theorem 4.6]{KPTT11}. In fact, this can also be seen from the concrete viewpoint that we have adopted in the above definition. Indeed, if $\phi: E \to E'$ and $\psi : F \to F'$ are complete order isomorphisms ({\em i.e.} completely isometric isomorphisms) then \cite[Proposition 8.1.5]{ER00} shows that $\phi \otimes \psi$ induces a complete order isomorphism from $E \otimes_\min F$ to $E' \otimes_\min F'$.


We now come to the main topic of this section, which is stable equivalence of operator systems. 
Let $\K$ denote the $C^*$-algebra of compact operators. 

\begin{defn}
We say that two operator systems $E$ and $F$ are {\em stably equivalent} if $E \otimes_\min \K$ and $F \otimes_\min \K$ are complete order isomorphic operator systems. 
\end{defn}
It is immediate that this is an equivalence relation. 
The advantage of working with the minimal tensor product should now become clear. In fact, if $E \subseteq \B(H)$ and $\K \equiv \K(K)$ for Hilbert spaces $H$ and $K$ we find that $E \otimes_\min \K(K) $ is
the closure of $E \otimes  \K_0(K)$ in $\B(H \otimes K)$ where $\K_0(K)$ denote finite-rank operators. 

We expect that the above notion of stable equivalence is related to a notion of Morita equivalence for operator systems, similar to what happens in the case of $C^*$-algebras \cite{BDR77}, operator algebras \cite{BM04} and operator spaces \cite{EK17}. It is an interesting open problem to develop such a notion and see how it reduces to the familiar notion of Morita equivalence. In any case, we record the following result.

\begin{prop}\label{prop.stabilize}  $(i)$~Let $E$ be a unital operator system. The $C^*$-envelope of the stabilization $E \otimes_\min \K$ is isomorphic to the stabilization $C^*_\env(E) \otimes \K$ of the $C^*$-envelope.\newline
$(ii)$~Let $E$ and $F$ be stably equivalent unital operator systems. Then $C^*_\env(E)$ and $C^*_\env(F)$ are stably equivalent $C^*$-algebras. 
\end{prop}
\proof $(i)$~Let $\kappa:E \to C^*_\env(E)$ be the $C^*$-envelope of the operator system $E$; and consider the map 
$
\alpha:=\kappa \otimes_\min {\rm id} : E \otimes_\min \K \to C^*_\env(E) \otimes_\min \K$. It makes sense since one can realize $
\kappa$ as an inclusion of  concrete operator systems. Let us show that $\alpha$ is a \sext-extension in the sense of Definition \ref{defnextension}. It is by construction a complete isometry and is completely positive. We need to show that  $
\alpha^\sharp  : (E \otimes_\min \K)^\sharp \to (C^*_\env(E) \otimes_\min \K)^\sharp$
is a complete order isomorphism with its range $R\subset (C^*_\env(E) \otimes_\min \K)^\sharp $. Let then $x=y+\lambda 1\in R$ where $y=\alpha(z)$ for some $z\in E \otimes_\min \K$ and $\lambda \in \R$. Assume that $x\geq 0$. This means that $\lambda \geq 0$ and that $y\geq -\lambda 1$ as concrete operators, \ie in $\B(H \otimes K)$. Let then $K_n\subset K$ be an increasing sequence of $n$-dimensional subspaces whose union is dense in $K$. Let $e_n\in \B(H \otimes K)$ be the orthogonal projection on  $H \otimes K_n$. One has $e_nye_n\geq -\lambda e_n$. Let then $z_n\in M_n(E)\subset   E \otimes_\min \K $ with $e_nye_n=\alpha(z_n)$. Since $
\kappa$ is an inclusion of  concrete operator systems one has $z_n+\lambda e_n\geq 0$. The element $z_n+\lambda 1\in (M_n(E))^\sharp $ is positive in the sense of Definition \ref{defn:unit-opsyst}. Indeed $M_n(E)$ is a unital operator system so its state space in the sense of positive linear functionals on $M_n(E)$ of norm 1 is the same as the ordinary state space  of positive functionals equal to $1$ on the unit $e_n$. Thus since $z_n+\lambda e_n\geq 0$ the positivity condition of Definition \ref{defn:unit-opsyst} holds.  We thus obtain a sequence of positive elements $v_n=z_n+\lambda 1\in (E \otimes_\min \K)^\sharp$ which is norm convergent and whose norm limit $v$ is positive and fulfills $\alpha^\sharp(v)=x$. The similar argument also applies when passing to matrices. 
 Thus we conclude that $C^*_\env(E) \otimes \K$ is a \sext-extension of $E \otimes_\min \K$. Let then $C^*(E \otimes_\min \K)=C^\sharp(E \otimes_\min \K)$ be the $C^*$-envelope of $E \otimes_\min \K$ which exists uniquely by Theorem  \ref{sext-envelope}. From the universal property of Definition \ref{defn-sext-envelope} it then follows that there is an ideal $J$ in $C^*_\env(E) \otimes \K$ such that $(C^*_\env(E) \otimes \K) /J  \cong C^*(E \otimes_\min \K)$. A closed two-sided ideal in the tensor product $A\otimes \K$ of a $C^*$-algebra $A$ by $\K$ is necessarily of the form $J_0\otimes \K$, where $J_0$ is the  closed ideal of $A$ defined by $J_0=\{a\in A \mid a\otimes k\in J, \forall k\in \K\}$. Thus we have here  $J=J_0\otimes \K$. By definition of $J$ the quotient map
$$
q : C^*_\env(E) \otimes \K \to (C^*_\env(E) \otimes \K)/J
$$
restricts to a complete isometry on $\kappa(E \otimes_\min \K) \subseteq C^*_\env(E) \otimes \K$. In particular, with $e\in \K$ a minimal projection,  the restriction of $q $ to $\imath_E(E) \otimes e$ is a complete isometry, and it agrees with the quotient morphism 
$q_0:C^*_\env(E)\to C^*_\env(E)/J_0$. Hence  $J_0$ is a boundary ideal for the $C^*$-envelope $\imath_E: E \to C^*_\env(E)$, so $J_0$ is contained in the \v{S}ilov boundary for the $C^*$-envelope which is $0$. Thus $J_0=0$ and the proof is complete. \newline
$(ii)$~Follows from $(i)$.
\endproof

\begin{corl}
Suppose $E$ and $F$ are unital $C^*$-algebras. If $E$ and $F$ are stably equivalent as operator systems, they are also stably equivalent as $C^*$-algebras.
\end{corl}
\proof
 This follows from Proposition \ref{prop.stabilize} $(ii)$ since $E = C^*_\env(E)$ and $F = C^*_\env(F)$. \endproof

\subsection{Propagation number as an invariant of operator systems}

Let $E$ be an operator system. For an integer $n>0$ one lets \Red{$E^{\circ n}$} be the norm closure of the linear span of products of $\leq n$ elements of $E$. It is an operator system. 
\begin{defn}
The {\em propagation number} ${\rm prop}(E)$ of the operator system $E$ is defined as the smallest integer $n$ such that $\imath_E(E)^{\circ n}\subseteq C_{\rm env}^*(E)$ is a $C^*$-algebra. 
  \end{defn}
When no such $n$ exists one lets ${\rm prop}(E)=\infty$. 

\begin{ex}\label{ex:disc-alg2} Returning to the example \ref{ex:disc-alg} of harmonic functions in the disk one has by \ref{ex:disc-alg1} that $C^*_\env(C_{\text{harm}}(\bar {\mathbb D})) \cong C(\overline{\mathbb{D}}) /C_0(\mathbb D) \cong C(S^1)$. The Poisson kernel 
$$
P(z, e^{it}):=\frac{1-\vert z\vert^2}{\vert e^{it}-z\vert^2}
$$
gives the canonical linear section $P:C(S^1)\to C_{\text{harm}}(\bar {\mathbb D})$ by the Poisson integral
$$
P(f)(z):=\frac{1}{2\pi}\int_{-\pi}^\pi P(z,e^{it})f(t)dt
$$
and (see \cite{Rudin} Theorem 11.8) this map is an isomorphism of operator systems. In particular the propagation number of $C_{\text{harm}}(\bar {\mathbb D})$ is equal to $1$. Note that this example shows that morphisms of operator systems between $C^*$-algebras are not in general morphisms of $C^*$-algebras.
\end{ex}

\begin{prop}
  The propagation number is invariant under completely isometric, complete order isomorphisms of operator systems.
  \end{prop}
\proof
This follows from the uniqueness property of the $C^*$-envelopes: given a complete order isomorphism $\phi:E \to F$ of two operator systems there is a commuting diagram:
$$
\xymatrix{
  E \ar[d]_{\imath_E} \ar[r]_{\phi} & F \ar[d]_{\imath_F}\\
  C_{\rm env}^*(E) \ar[r]_{\tilde\phi} & C_{\rm env}^*(F)
}
$$
where $\tilde \phi$ is a $*$-isomorphism. Hence if ${\rm prop}(E)=n$ then we obtain that
$$
(\imath_F(F))^{\circ n} = \left(\imath_F(\phi(E))\right)^{\circ n} = \left(\tilde \phi (\imath_E(E))\right)^{\circ n} = \tilde \phi \left((\imath_E(E))^{\circ n}\right) = C_{\rm env}^*(F),
$$
and ${\rm prop}(F) \leq n = {\rm prop}(E)$. Similarly we find ${\rm prop}(E) \leq {\rm prop}(F)$ which completes the proof. 
\endproof

\begin{prop}
The  propagation number is invariant under stable equivalence, {\em i.e.}, for any unital operator system $E$ we have
$$
{\rm prop}(E) = {\rm prop}(E \otimes_\min \K)
$$
where $\K$ is the $C^*$-algebra of compact operators.
\end{prop}
\proof
Suppose that ${\rm prop}(E)=n$. Then
\begin{align*}
\imath_{E \otimes_\min \K} (E \otimes_\min \K)^{\circ n} =
\imath_{E \otimes_\min \K} (E \otimes_\alg \K)^{\circ n}  = \imath_{E}(E)^{\circ n} \otimes \K = C^*_\env(E) \otimes \K.
\end{align*}
We conclude that ${\rm prop}(E \otimes_\min \K) \leq n = {\rm prop}(E)$. 

In the other direction, let us now assume that ${\rm prop} (E \otimes_\min \K) = n$. Then by Proposition \ref{prop.stabilize}, one has $C^*_\env(E \otimes \K) =C^*_\env(E) \otimes \K $ and thus the latter is the norm closure of the linear span of the products of at most $n$ elements of  $E \otimes_\alg \K$. Let $x\in C^*_\env(E)$ and $e\in \K$ a minimal projection. One can thus approximate $x\otimes e$ by a finite sum of elements of the form 
$$
m=(a_1\otimes k_1)\cdots (a_n\otimes k_n), \qquad a_j\in E, \ k_j\in \K.
$$
One thus obtains a good approximation of $x$ by the sum of the 
$$
\lambda\, a_1\cdots a_n, \qquad e k_1\cdots k_n e=\lambda e
$$
 from which it follows that ${\rm prop}(E) \leq n = {\rm prop}(E \otimes_\min \K)$.
\endproof

\section{Spectral truncations}
\label{sect:truncations}
The basic paradigm in noncommutative geometry is given by a so-called {\em spectral triple} $(\A,\H,D)$, combining a $*$-algebra of bounded operators on Hilbert space $\H$ with a self-adjoint operator $D$ with compact resolvent and bounded commutators $[D,a]$ for all $a \in \A$. The typical example is given by a compact Riemannian spin manifold $M$ where $\A= C^\infty(M)$ and $D$ is the Dirac operator acting on $L^2$-spinors. In fact, it is possible \cite{C08} to reconstruct a Riemannian spin manifold from any spectral triple $(\A,\H,D)$ satisfying certain conditions, including commutativity of $\A$. 

As we explained in the introduction, our goal is to extend this approach of geometry to cases where (possibly) only part of the spectrum of $D$ is available. Most naturally, we may consider a spectral triple $(\A,\H,D)$ where we take a cutoff of the operator $D$, described by means of a spectral projection $P$, projecting onto a finite number of the eigenvectors of $D$. Clearly, such an operator commutes with $D$ and, in fact, any operator that commutes with $D$ also commutes with $P$ so that the truncation respects the group of isometries of the geometry. The operator $D$ is thus replaced by $PD = PDP = DP$ acting as an operator on the Hilbert space $P \H$.

Clearly, the $*$-algebra $\A$ does not act on $P\H$ any more. However, we may form the space $P\A P$ of operators which is invariant under the involution. In other words, $P\A P$ is an operator system in $\B(P\H)$. We thus come to consider the triple $(P\A P, P \H, PDP)$.

The advantage of this spectral description is that we now have the possibility to work with a possibly finite-dimensional truncation of the spectral geometry, while keeping all the isometries of the original spectral triple intact. Indeed, the isometries of the latter are given by unitaries $U$ such that $U DU^*=D$ and $U \A U^* = \A$. Now, since $P$ commutes with such $U$, it follows that $U$ also acts unitarily on $P\H$ while it maps $P\A P$ and $D$ to itself.

We thus arrive at the following generalization of spectral triples with $*$-algebras of bounded operators replaced by operator systems.

\begin{defn} An {\em operator system spectral triple} is a triple $(\E,\H,D)$ where $\E$ is a dense subspace of a (concrete) operator system $E$ in $\mathcal B(\H)$, $\H$ is a Hilbert space and $D$ is a self-adjoint operator in $\H$ with compact resolvent and such that $[D,T]$ is a bounded operator for all $T \in \E$.
\end{defn}

Since states are perfectly defined on operator systems, the above triple induces a (generalized) distance function on $\cS(E)$ by setting
\begin{equation}
\label{eq:dist}
d(\phi,\psi) = \sup_{x \in \E} \left\{ | \phi(x) - \psi(x)|: \|[D, x]\| \leq 1 \right\}.
\end{equation}
If $\E = \A$ is a $*$-algebra then this reduces to the usual distance function \cite{C94} on the state space of the $C^*$-algebra $A = \overline \A$ . It also agrees with the definition of quantum metric spaces based on order-unit spaces given in \cite{Rie00,Ker03,Lat16}. It has been studied for truncations in \cite{ALM14,GS19a,GS19b}. The properties of this metric distance function and the notions of Gromov--Hausdorff convergence it gives rise to will be studied elsewhere. We will here focus on a detailed analysis on the topological structure. We start with the simplest case given by spectral truncations of the circle. As we will find the theory is extremely rich, in fact, already in a spectral truncation of the circle of rank 3.

\subsection{Spectral truncation of the circle}
\label{sect:trunc-circle}
We consider $C^\infty(S^1) \subset C(S^1)$ acting as bounded multiplication operators on $L^2(S^1)$. 
An orthonormal basis of $L^2(S^1)$ is given in terms of Fourier theory: $e_k (x)= e^{ikx}$ for $k \in \Z$. These are of course eigenvectors for the Dirac operator $-i d/dx$ on the circle. We consider a spectral truncation defined by the orthogonal projection $P_{n}$ onto span$_\C\{ e_{1},e_{2},\ldots , e_{n}\}$ for some fixed $n > 0$. We will also write simply $P=P_{n}$. The space $P C(S^1)P$ is an operator system and 
an arbitrary element $T=PfP$ in $\Toep{n} = PC(S^1)P$ can be written as the following matrix with respect to the orthonormal basis span$_\C\{ e_{k}\}_{k=1}^n$:
\begin{equation}\label{toeplitz}
P f P \sim \begin{pmatrix}
a_0 & a_{-1} & a_{-2} &a_{-3} & \ldots& a_{-n+1} \\
a_1 & a_0 & a_{-1} & a_{-2}&\ldots &a_{-n+2}\\ 
a_2 & a_1 & a_0 & a_{-1} & \ldots & a_{-n+3}\\
a_3 & a_2 & a_1 & a_{0} & \ldots & a_{-n+4}\\
\vdots & \vdots & \vdots & \vdots & \ddots & \vdots\\
a_{n-1} & a_{n-2} & a_{n-3} & a_{n-4}  & \cdots & a_{0}
\end{pmatrix}
\end{equation}
in terms of the Fourier coefficients $\{a_{n}\}_{n \in \Z}$ of $f \in C^\infty(S^1)$. Such matrices with constants along all diagonals are called {\em Toeplitz matrices}. Hence the spectral truncation $P_n C(S^1)P_n$ is given by the {\em Toeplitz operator system} of all $n \times n$ Toeplitz matrices. A fully general analysis of the Toeplitz system will be postponed to the next section, including the $C^*$-envelope, the propagation number, the extreme rays in the cone of positive elements, the pure state spaces, {\em et cetera}.

In the next subsection we anticipate this discussion and lift the curtain on the rich structure that is found already in the simplest non-trivial case, namely when $n=3$.

\subsection{State space of the truncated circle ($n=3$)}
We shall proceed by discussing in details the simplest non-trivial case 
 which is $n=3$, but before entering in the details we list some properties which ought to be shared with the general case:
 
 \begin{itemize}
 \item The extreme rays of the positive cone $O_+$ of the operator space form a circle.  More precisely they are proportional to rank one self-adjoint idempotents  which belong to $O$. They remain extreme rays in the cone of positive matrices.  
 \item  The boundary $\partial O_+$ of the base of the positive cone $O_+$ is the closure of a component of the complement of singular points in a  rational algebraic hypersurface $H$.
 \item The extreme points of the base of the positive cone $O_+$ are singular points of the above hypersurface $H$.

 \item The boundary $\partial \cS$ of the state space is a component of the complement of singular points in a  rational algebraic hypersurface $K$.

  \item The extreme points $E$ of  the state space coincide with the singular points of the above hypersurface $K$  and form the quotient of a torus by the symmetric group.

 \item Both hypersurfaces $H$ and $K$ are the union of a pencil of lines obtained from pairs of points of the singular set.
  \end{itemize}
In the case of the base of the cone $O_+$ the pencil of lines is of dimension $2$ (for $n=3$) and its elements are parametrized by arbitrary pairs of points of the curve $\Gamma$ of extreme points. In the case of the state space the singular set of $K$ is two dimensional and the pencil of lines is formed of lines joining an arbitrary point of the singular set with a precisely defined corresponding point. The pencil of lines is also two dimensional.

 \subsubsection{The positive cone $O_+$ and its extreme rays}
 The self-adjoint elements of the truncated operator space $P_3C(S^1)P_3$ are matrices of the form  
$$
\left(
\begin{array}{ccc}
 u & a-i b & c-i d \\
 a+b i & u & a-i b \\
 c+d i & a+b i & u \\
\end{array}
\right),\ \ (a,b,c,d,u)\in \R^5.
$$
They form a $5$-dimensional real vector space $O$ and what matters is to understand the positive cone  $O_+$ for the operator norm. The state space is then obtained using the dual cone and intersecting with the hyperplane $\phi(1)=1$. The coefficients of the characteristic polynomial of the above matrix are the following
\begin{multline*}
\bigg\{-2 a^2 (c-u)-4 a b d+2 b^2 (c+u)+u \left(c^2+d^2-u^2\right),\\-2 a^2-2 b^2-c^2-d^2+3 u^2,-3 u,1\bigg\}
\end{multline*}
and they give a map $\psi:O\to \R^3$. 
\begin{lem}
\label{lma:cone-det0}
  The cone $O_+$ of positive elements is the closure of the open component of $(0,0,0,0,1)$ in the complement of the hypersurface 
$$
H:=\{(a,b,c,d,u)\mid 2 a^2 c-2 a^2 u+4 a b d-2 b^2 c-2 b^2 u-c^2 u-d^2 u+u^3=0\}
$$	
\end{lem}
\proof The elements of the cone $O_+$ are the  elements of $O$ whose eigenvalues are positive. Since all elements of $O$ are self-adjoint matrices, their eigenvalues are real. The cone $O_+$ is convex and is the closure of its interior which consists of 
matrices $A$ whose eigenvalues are strictly positive. For $A\in O_+$ is strictly positive, the segment joining $A$ to the identity matrix $1$ stays inside $O_+$ and thus $A$ belongs to the open component of $(0,0,0,0,1)$ in the complement of the hypersurface $H$. Conversely on a path in the complement of the hypersurface $H$ joining $1$ to $A$ the eigenvalues remain positive since they vary continuously and cannot vanish as their product is given by the determinant of the matrix which is
$$
2 a^2 c-2 a^2 u+4 a b d-2 b^2 c-2 b^2 u-c^2 u-d^2 u+u^3
$$
equal to the cubic polynomial which defines $H$. \endproof  
Since the trace of a matrix is the sum of its eigenvalues one knows that it is $>0$ on the cone  $O_+$ and hence the natural basis of the cone $O_+$ is its intersection $O_{+,1}$ with the hyperplane $u=1$. 
The hypersurface $H$ is then determined by the three dimensional zero set $Z$ of the polynomial 
$$
\delta(a,b,c,d)=2 a^2 (c-1)+4 a b d-2 b^2 (c+1)-c^2-d^2+1
$$
and $O_{+,1}$ is the closure of the open component $B$ of $(0,0,0,0)$ in the complement of $Z$. The singularities of the hypersurface $Z$ correspond to the points of $Z$ at which the gradient of $\delta$ vanishes. This gradient is given by 
$$
\nabla\delta=\left\{4 a (c-1)+4 b d,4 a d-4 b (c+1),2 \left(a^2-b^2-c\right),4 a b-2 d\right\}
$$
and the points of $Z$ on which $\nabla\delta$ vanishes form the rational curve 
$$
\Gamma:=\{(a,b,c,d)\mid a^2+b^2=1, \ c=-1+2a^2, \ d=2 ab\}.
$$
It is parametrized in the form 
$$
(a,b,c,d)=(\cos(x),\sin(x),\cos(2x),\sin(2x))=\gamma(x)\in \Gamma.
$$
\begin{figure}
\centering
\includegraphics[scale=.5]{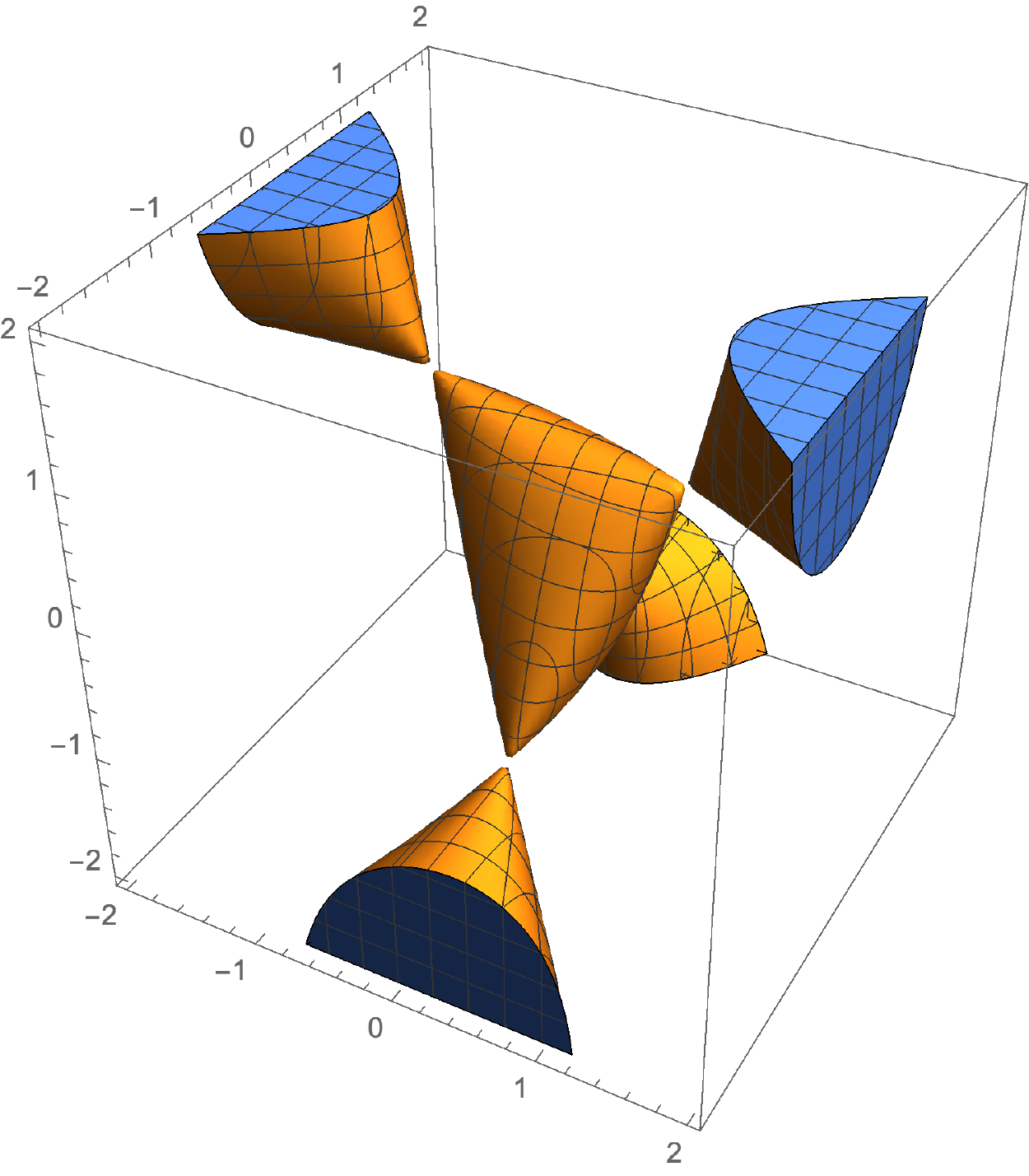}
\caption{The hypersurface $Z$ intersected with the hyperplane $d=-4/10$. One can see $6$ components of the complement, the central one is convex}
\end{figure}
\begin{lem}\label{mapsig} The boundary $\partial B$ of  $B$  is the range of  
$\sigma:\T^2\times [0,1]\to Z\subset \R^4$ 
$$
\sigma(x,y,s):=s \gamma(x)+(1-s)\gamma(y).
$$
\end{lem}
\proof One first checks that $\sigma(x,y,s)\in Z$. One has
$$
\delta(\sigma(x,y,s))=(s \sin (2 x)+(1-s) \sin (2 y))^2-2 (s \cos (2 x)+$$
$$+(1-s) \cos (2 y)-1) (s \cos (x)+(1-s) \cos (y))^2+ 
(s \cos (2 x)+(1-s) \cos (2 y))^2$$
$$-4 (s \sin (x)+(1-s) \sin (y)) (s \sin (2 x)+(1-s) \sin (2 y)) (s \cos (x)+(1-s) \cos (y))$$
$$+2 (s \sin (x)+(1-s) \sin (y))^2 (s \cos (2 x)+(1-s) \cos (2 y)+1)-1.
$$
After expanding in powers of $s$ one finds that all coefficients vanish identically. 

Thus the range of $\sigma$ is contained in the hypersurface $Z$ and we now verify that it is contained in the boundary of the component $B$ of the complement of $Z$. Let us show that any element $\gamma(x)\in \Gamma$ belongs to the boundary of $B$. One considers the path given (with parameter $t\in [0,1]$) by
$$
t\mapsto p(t,x):= (1-t)/3 \left(\gamma(x)+\gamma(x+2\pi/3)+\gamma(x+4\pi/3) \right)+t\gamma(x).
$$
One finds that independently of $x\in [0,2\pi]$ one has $\delta(p(t,x))=(t-1)^2 (2 t+1)$, while one has $p(0,x)=0$ and $p(1,x)=\gamma(x)$. This shows that $\gamma(x)\in \Gamma$ belongs to the boundary of $B$. It follows by convexity of $B$ that the range of $\sigma$ is contained in the boundary of $B$. 

The minors of the Jacobian of the map $\sigma$ are 
$$
\left(
\begin{array}{c}
 -8 (s-1) s \sin ^4\left(\frac{x-y}{2}\right) \sin (x+y) \\ 8 (s-1) s \cos (x+y) \sin ^4\left(\frac{x-y}{2}\right) \\ 16 (s-1) s \sin ^4\left(\frac{x-y}{2}\right) (\sin (x)+\sin (y)) \\ -16 (s-1) s (\cos (x)+\cos (y)) \sin ^4\left(\frac{x-y}{2}\right)\\
\end{array}
\right).
$$
They all vanish if and only if $(s-1) s \sin ^4\left(\frac{x-y}{2}\right)=0$. This holds only if $s\in \{0,1\}$ or if the pair $(x,y)$ belongs to the diagonal $\Delta\subset \T^2$. In both cases the critical value is on the curve $\Gamma$. The complement of $\Gamma$ in $\partial B$ is connected since $\partial B$ is the boundary of the convex body $B$ (the boundedness of $B$ follows from the boundedness of positive matrices of fixed trace) and is thus homeomorphic to a three dimensional sphere which the curve $\Gamma$ cannot disconnect. Thus the intersection of the range of $\sigma$ with the complement of $\Gamma$ in $\partial B$ is both open and closed (since the domain $\T^2\times [0,1]$ is compact and critical values lie in $\Gamma$) and is thus equal to the complement of $\Gamma$ in $\partial B$. Since  $\Gamma$ lies in the range of $\sigma$ one gets the required surjectivity. \endproof 
\begin{prop} \label{extrepts} The extreme points of the convex set $O_{+,1}$ form the curve $\Gamma$. 	
\end{prop}
\proof The extreme points of $O_{+,1}$ belong to the boundary $\partial B$ and hence to the range of the map $\sigma$ of Lemma \ref{mapsig}. The formula defining $\sigma$ shows that provided $\gamma(x)\neq \gamma(y)$ all the non-trivial convex combinations are not extreme points and thus the extreme points of $\partial B$ are contained in $\Gamma$. To see that any element $\gamma(x)\in \Gamma$ is an extreme point we note that the Toeplitz matrix associated to $\gamma(x)$ is of the form 
$$
\begin{pmatrix} 1 & e^{-ix}  & e^{-2ix}  \\
 e^{ix} & 1 & e^{-ix}  \\
 e^{2ix}  & e^{ix}  & 1 \\
\end{pmatrix}=
\begin{pmatrix} 1 & 0  &0  \\
 0 & e^{ix} & 0  \\
 0  & 0  & e^{2ix} \\
\end{pmatrix}
\begin{pmatrix}1 & 1  & 1  \\
 1 & 1 & 1  \\
 1  & 1  & 1 \\
\end{pmatrix}
\begin{pmatrix}1 & 0  &0  \\
 0 & e^{-ix} & 0  \\
 0  & 0  & e^{-2ix} \\
\end{pmatrix}
$$
which is three times a rank one idempotent and is already an extreme point among positive matrices with fixed trace. \endproof 

\subsubsection{The state space and the pure states}
The tangent hyperplane to $B$ at any of the points $\sigma(x,y,s)$ for $s\in (0,1)$ is 
governed by the equation obtained by differentiating $\delta$ at such points which gives, up to the overall factor $8 (s-1) s \sin ^2\left(\frac{x-y}{2}\right)$ the vector
$$\{2 (\cos (x)+\cos (y)),2 (\sin (x)+\sin (y)),-\cos (x+y),-\sin (x+y)\}  . $$
The associated linear form evaluated at any of the points $\sigma(x,y,s)$ takes the value $\cos (x-y)+2$.  Replacing $x\mapsto x+\pi$ and $y\mapsto y+\pi$ this provides us with a two-parameter family of supporting hyperplanes for the cone $O_+$ given by the equation $
L(a,b,c,d,u)=0$ where 
$$L(a,b,c,d,u)=2 a (\cos (x)+\cos (y))+2 b (\sin (x)+\sin (y))+$$
$$+ c \cos (x+y)+ d \sin (x+y)+ u (\cos (x-y)+2).
$$
In order to obtain a state one normalizes $L$. This gives us the following map $\epsilon:\T^2\to O^*$ from the two torus $\T^2$ to the state space $\cS$ of $O$
$$
\epsilon(x,y)(a,b,c,d,u):=\frac{1}{\cos (x-y)+2}L(a,b,c,d,u).
$$
Since the circle $x^2+y^2=1$ is a rational curve with rational parametrization given by 
$$
x=\frac{2 t}{t^2+1}, \ \ y=\frac{1-t^2}{t^2+1}
$$
one obtains a rational parametrization of the range of the map $\epsilon$ in the form
$$
W=\frac{4 \left(t^2 v+t v^2+t+v\right)}{t^2 \left(3 v^2+1\right)+4 t v+v^2+3},X=\frac{4-4 t^2 v^2}{3 t^2 v^2+t^2+4 t v+v^2+3},$$
$$Y=\frac{t^2 \left(-\left(v^2-1\right)\right)+4 t v+v^2-1}{t^2 \left(3 v^2+1\right)+4 t v+v^2+3},Z=\frac{2 \left(t^2 (-v)-t v^2+t+v\right)}{t^2 \left(3 v^2+1\right)+4 t v+v^2+3}.
$$
By elimination of the variables $(t,v)$ one obtains that the range is contained in the following two quartic hypersurfaces 
$$
W^2 X^2+2 W^2 Z^2-4 W X Z+2 X^2 Z^2=0
$$
and 
$$
W^2 \left(X^2+4 Z^2\right)+4 Z^2 \left(X^2+4 \left(Y^2+Z^2-1\right)\right)=0.
$$
\begin{figure}
\centering
\includegraphics[scale=.5]{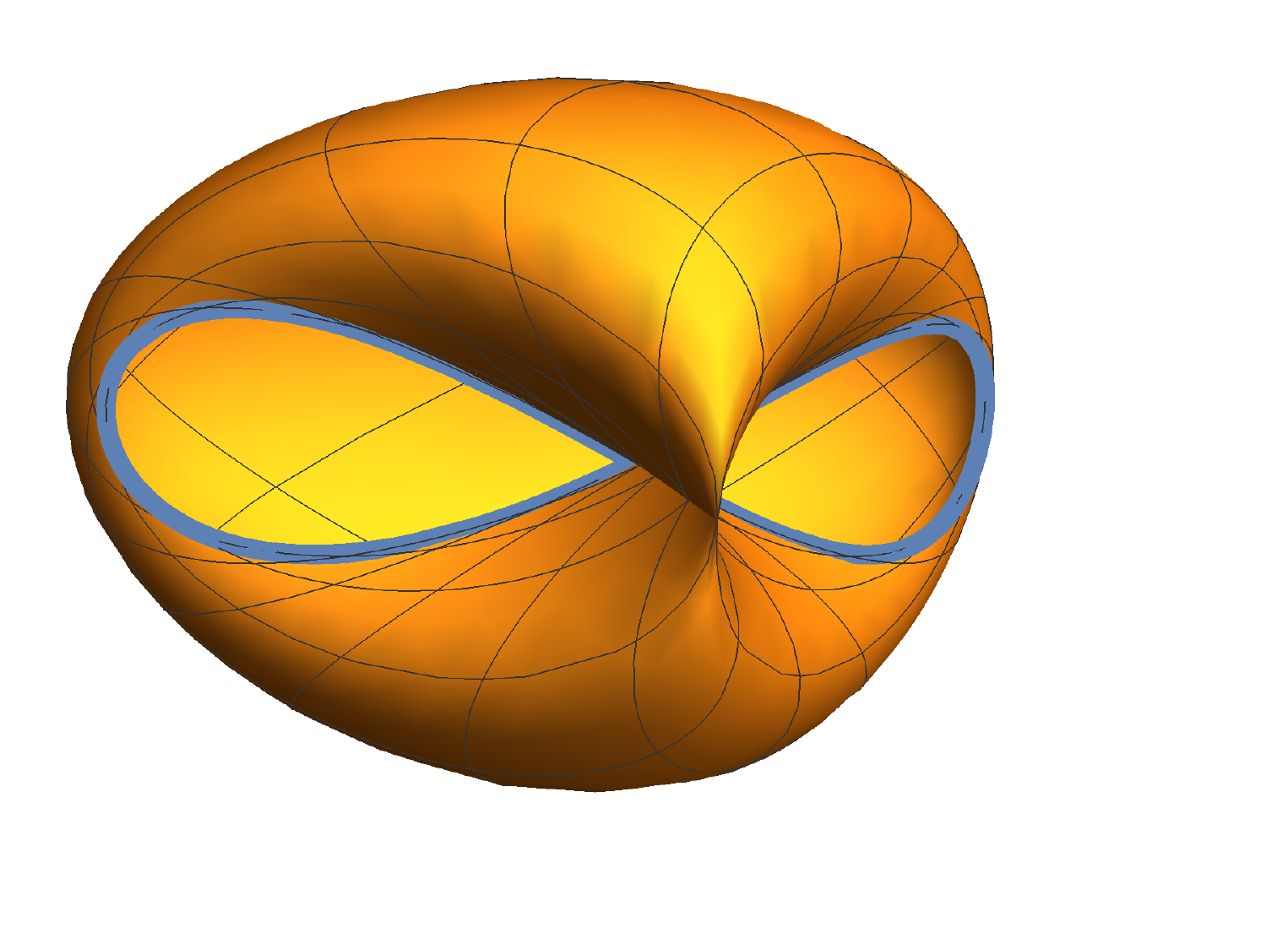}
\caption{The surface $\Sigma$ in $\R^3$.}\label{surface}
\end{figure} 
These equations allow one to express $W$ as follows
$$
W=\frac{-X^4 Z-8 X^2 Y^2 Z-8 X^2 Z^3+8 X^2 Z-16 Y^2 Z^3-16 Z^5+16 Z^3}{2 X \left(X^2+4 Z^2\right)}
$$
and one just needs to understand the surface $\Sigma$ in $\R^3$ given by the parametrization by $(X,Y,Z)(t,v)$ or by the equation
\begin{equation}\label{equmoeb}
X^4+8 X^2 Y^2+8 X^2 Y+8 X^2 Z^2+16 Y^2 Z^2+16 Z^4-16 Z^2=0.
\end{equation}
It is represented in Figure \ref{surface}.

The boundary $\partial\cS$ of the state space $\cS$ is contained in an algebraic hypersurface which is determined as follows. The linear form 
$$
L(a,b,c,d,u)=aW+bX+cY+dZ+u
$$
belongs to the state space if and only if it takes positive values on the extreme points of the base of the cone $O_+$. This means using Proposition \ref{extrepts} and the rational parametrization of the curve $\Gamma$ that the product 
$$
(-t^4 X-t^4 Y+t^4+2 t^3 W-4 t^3 Z+6 t^2 Y+2 t^2+2 t W+4 t Z+X-Y+1)(1+t^2)^{-2}
 $$
 only takes positive values. This thus means that the polynomial first factor $P(t)$ is positive for all $t\in \R$. If $X+Y=1$, $P(t)$ is of degree $3$ and hence cannot be positive for all $t\in \R$  unless the coefficient $2W-4Z$ of $t^3$ vanishes.  Since the boundary $\partial\cS$ is topologically a sphere we can ignore this codimension $2$ situation. One then has $X+Y<1$ \ie  the coefficient of $t^4$ is positive. Then $P(t)$ is positive for all $t\in \R$ if and only if it takes positive values where its derivative vanishes:
 $$
 P'(t)=0\Rightarrow P(t)\geq 0.
 $$
  This shows that in the boundary $\partial\cS$ one has a common root for $P$ and $P'$. Thus $\partial\cS$
 is contained in the zero set of the discriminant which is the following polynomial 
$$
d(W,X,Y,Z)=W^6+3 W^4 X^2+15 W^4 Y^2-18 W^4 Y-12 W^4 Z^2-W^4+108 W^3 X Y Z$$ $$-36 W^3 X Z+3 W^2 X^4-78 W^2 X^2 Y^2+84 W^2 X^2 Z^2-2 W^2 X^2+48 W^2 Y^4-144 W^2 Y^3$$ $$+96 W^2 Y^2 Z^2+80 W^2 Y^2-144 W^2 Y Z^2+16 W^2 Y+48 W^2 Z^4+80 W^2 Z^2-108 W X^3 Y Z$$ $$-36 W X^3 Z-288 W X Y^2 Z-288 W X Z^3+32 W X Z+X^6+15 X^4 Y^2+18 X^4 Y-12 X^4 Z^2$$ $$-X^4+48 X^2 Y^4+144 X^2 Y^3+96 X^2 Y^2 Z^2+80 X^2 Y^2+144 X^2 Y Z^2-16 X^2 Y+48 X^2 Z^4$$ $$+80 X^2 Z^2-64 Y^6-192 Y^4 Z^2+128 Y^4-192 Y^2 Z^4+256 Y^2 Z^2-64 Y^2-64 Z^6+128 Z^4-64 Z^2.$$

\begin{figure}
\centering
\includegraphics[scale=.5]{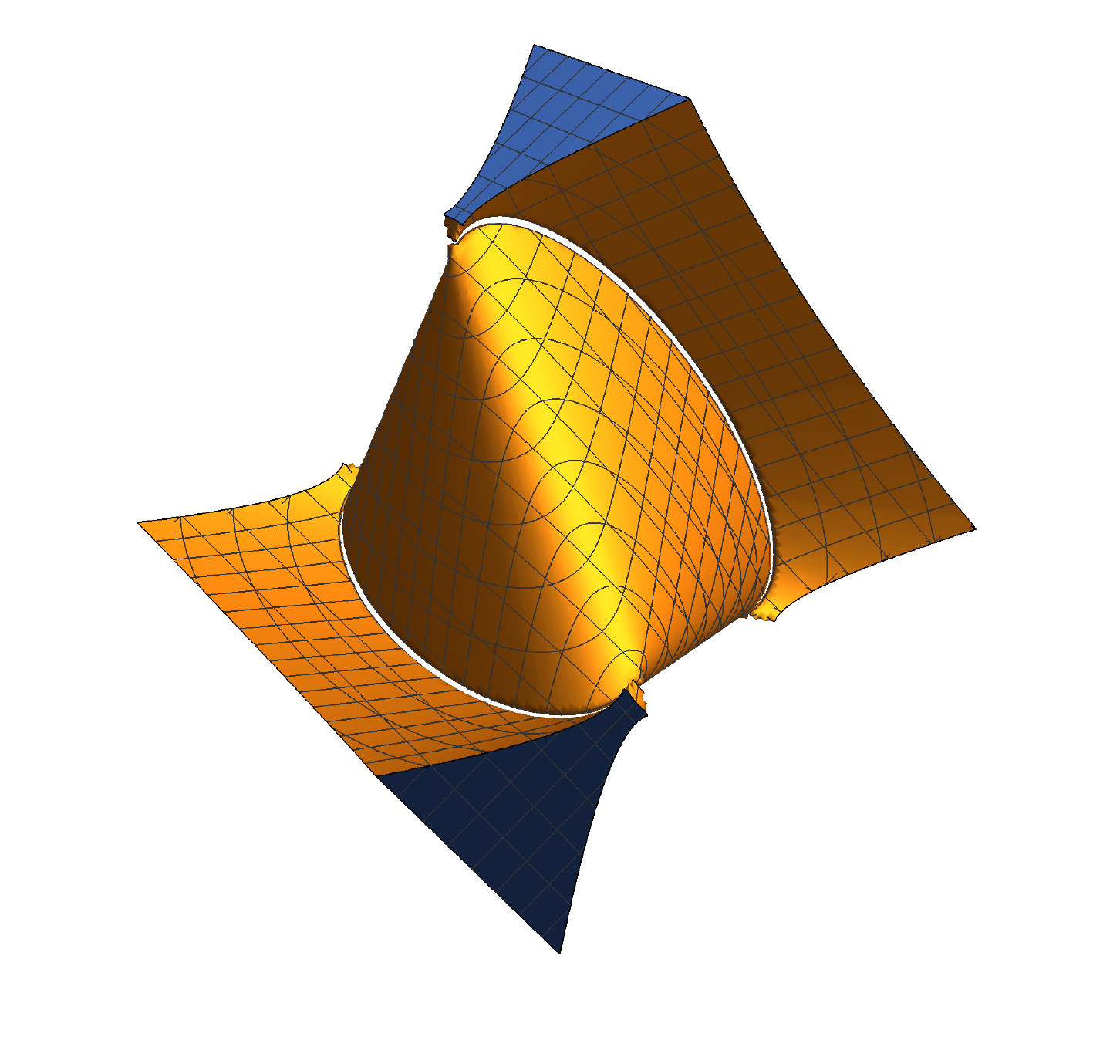}
\caption{Intersection of the hypersurface $d(W,X,Y,Z)=0$ with $Z=0$}\label{hypersurface}
\end{figure}
\begin{figure}
\centering
\includegraphics[scale=.5]{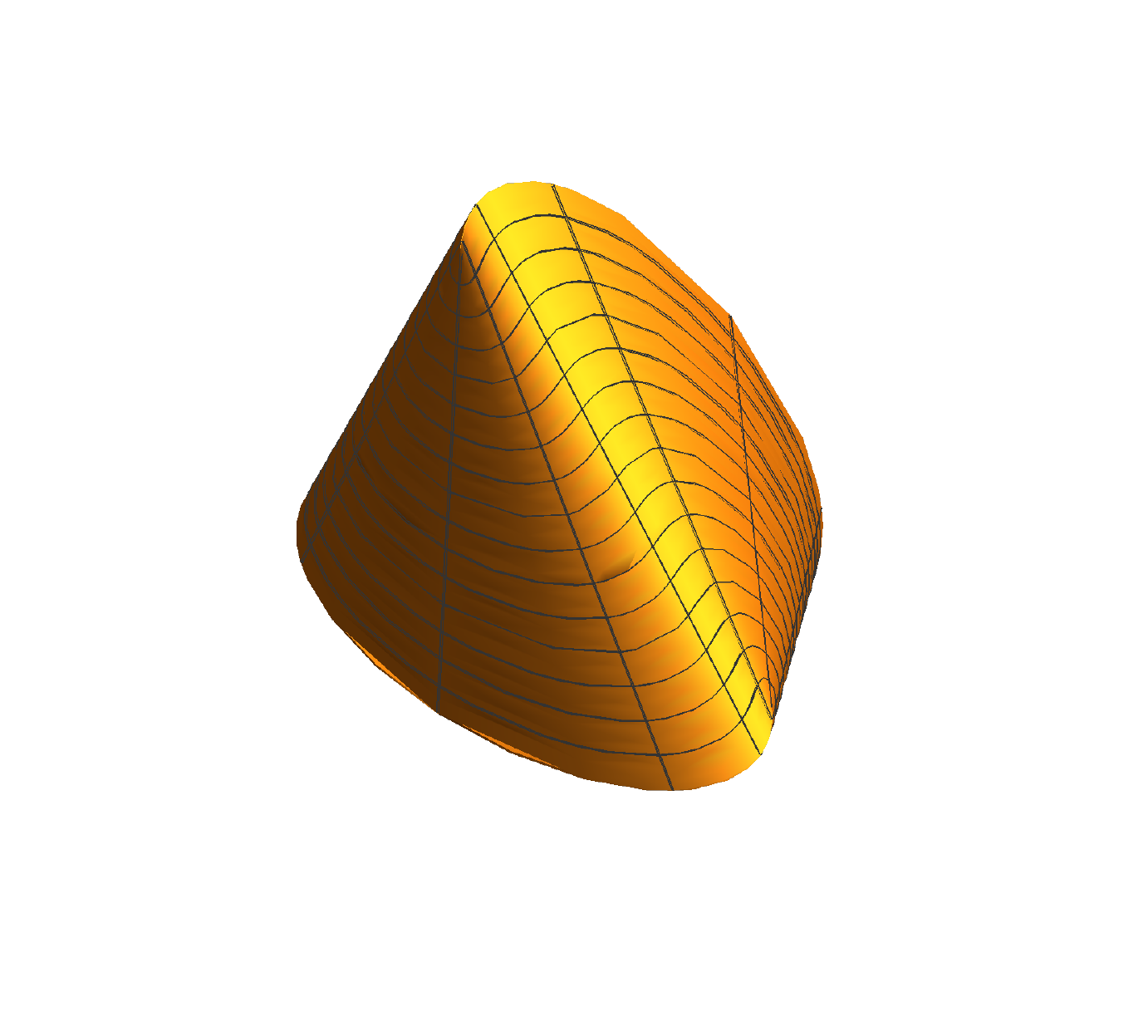}
\caption{State space intersected  with $Z=0$}\label{hypersurface1}
\end{figure} 
We  thus get the following  
\begin{lem} The boundary $\partial \cS$ of the state space $\cS$ is contained in the hypersurface  $K:=\{(W,X,Y,Z)\mid d(W,X,Y,Z)=0\}$.
\end{lem}
We are now ready to determine the extreme points of the state space of the operator system $O$.

\begin{thm} \label{thmextremepoints}The map $\epsilon:\T^2\to O^*$ from the two torus $\T^2$ to the state space $\cS$ of $O$ defines a surjective double cover of the set $E$ of extreme points, and $E$ is a M\"obius strip with boundary the image $\epsilon(\Delta)$ of the diagonal. \newline
The M\"obius strip $E$ lies in the singular set of the hypersurface $d(W,X,Y,Z)=0$.\newline
The following map $\beta:\T^2\times [0,1]\to O^*$ is a surjection on the boundary $\partial \cS$ of the state space $\cS$
$$
\beta(x,y,s):= s\epsilon(x,y)+(1-s)\epsilon(x,y+\pi).
$$
\end{thm}
\proof The sum of squares of the numerators of the minors of the Jacobian of the map $\epsilon$ simplifies to 
\begin{multline*}
  2 \left(3 \sin \left(\frac{x-y}{2}\right)+\sin \left(\frac{3 (x-y)}{2}\right)\right)^2  \\
  \times (8 \cos (x-y)+4 \cos (2 (x-y))+\cos (3 (x-y))+7)
\end{multline*}
which only depends upon $x-y$ and whose graph, as a function of the single variable $x-y$, is shown in Figure \ref{grapheps}.
\begin{figure}
\centering
\includegraphics[scale=.6]{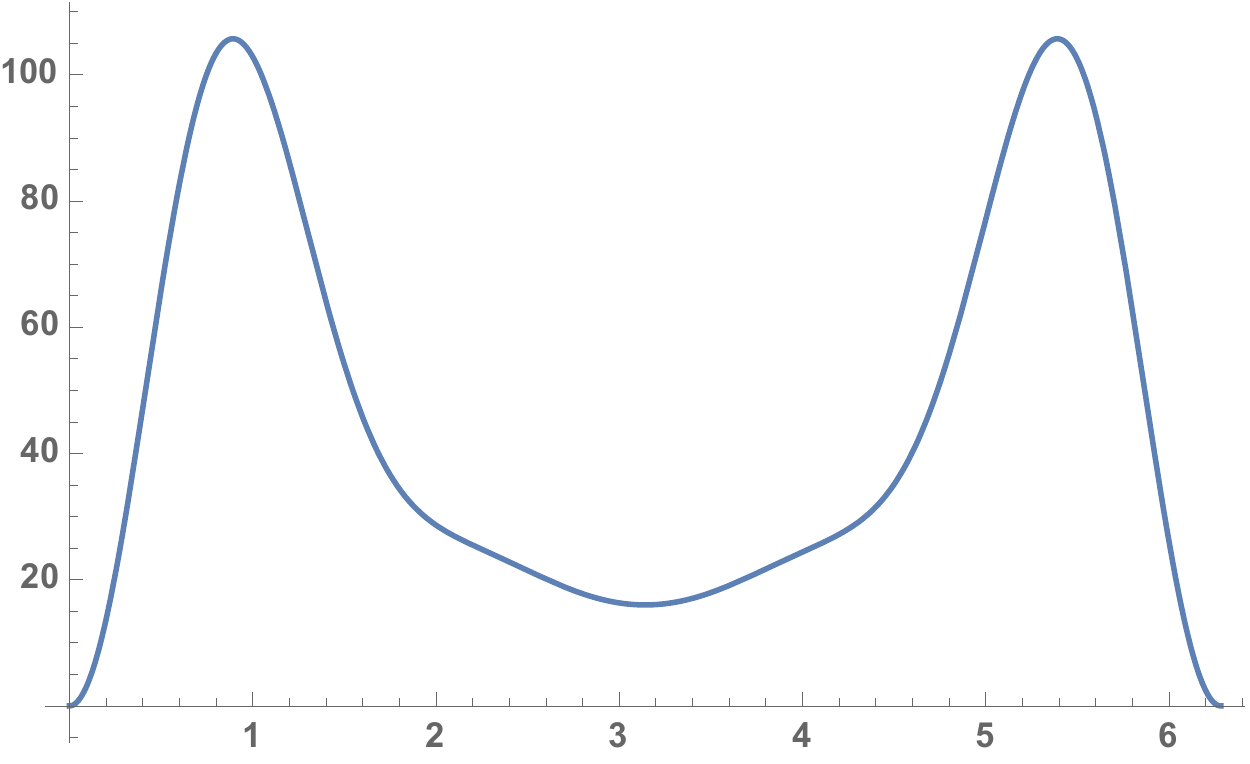}
\caption{The graph of the numerator of the size$^2$ of Jacobian}\label{grapheps}
\end{figure}
 This shows that the  map $\epsilon:\T^2\to O^*$ is an immersion except on the diagonal $\Delta$. 
One has $\epsilon(x,y)=\epsilon(y,x)$ so that $\epsilon$ passes to the quotient of $\T^2$ by the flip $s(x,y):=(y,x)$. Let us show that one thus obtains an injection of $\T^2/s$ in the state space $\cS$. One rewrites 
the components of $\epsilon(x,y)$ as a pair of complex numbers in terms of $u=e^{ix}$ and $v=e^{iy}$ using 
$$
2\cos(x-y)=u/v+v/u, \ uv= \cos(x+y)+i \sin(x+y),
$$
so that $\epsilon(x,y)$ determines the two complex numbers
$$
A=(u+v)/(4+u/v+v/u), \ \ B=uv/(4+u/v+v/u).
$$
One has $A/B=1/u+1/v$ and moreover 
$$
B-A^2=\frac{2 u^3 v^3}{\left(u^2+4 u v+v^2\right)^2}, \ \ B^2=\frac{u^4 v^4}{\left(u^2+4 u v+v^2\right)^2}
$$
which determines the product $uv$ as $uv=2B^2/(B-A^2)$. Now both $u,v$ are complex numbers of modulus one and one knows 
$$
1/u+1/v=A/B, \ \ †1/(uv)=\frac 12(B-A^2)/B^2
$$
so that $u,v$ are determined up to the flip. On the diagonal $\Delta\subset \T^2$ the map $\epsilon$ simplifies to
$$
\epsilon(x,x)= \left(\frac{4 \cos (x)}{3},\frac{4 \sin (x)}{3},\frac{1}{3} \cos (2 x),\frac{1}{3} \sin (2 x)\right)
$$
which is a smooth embedding of $S^1$ in $O^*$. This curve $\epsilon(\Delta)$ is the boundary of 
the surface $\Sigma=\epsilon(\T^2)$ which by the above discussion is the quotient of the two torus $\T^2$ by the flip, or equivalently the space of un-ordered pairs of complex numbers $u,v$ of modulus one. Such pairs form a M\"obius strip since setting $uv=\lambda^2$ the pair is determined by the element 
$$\frac{u+v}{2\lambda}\in [-1,1]$$ 
which depends up to sign on the choice of the square root $\lambda$. Thus the monodromy along the circle is the map $x\mapsto -x$ on the interval and the total space is  a M\"obius strip. One checks by direct computation that the gradient of $d$ vanishes on $\Sigma$ so that the latter is contained in the singular set of the hypersurface $K$. It remains to show that $\Sigma=\epsilon(\T^2)$ is the set $E$ of extreme points of the state space $\cS$. Let us first show that all other points of the boundary $\partial \cS$ are not extreme. This will follow if we show, as stated in the theorem, that the map $\beta$ is a surjection to the boundary $\partial \cS$. One first checks by direct computation that $d(\beta(x,y,s))=0$ for all $x,y\in \T$ and $s\in \R$. Moreover the computation of the gradient of $d$ on $\beta(x,y,s)$ gives the expression 
$$
\frac{32 (s-1) s (12 s \cos (x-y)-6 \cos (x-y)-\cos (2 (x-y))-5)^3}{\left(\cos ^2(x-y)-4\right)^4}
$$
multiplied by the non-zero vector $(\cos (x),\sin (x),-\cos (2 x),-\sin (2 x))$. One can thus check that $\beta(x,y,s)$ is a non-singular point of the hypersurface $K$ when $s\in (0,1)$ while, as seen above, it is a singular point for $s=0$ and $s=1$.

 The numerator of the sum of squares of minors of the Jacobian of $\beta$ simplifies to 
$$
8 \big((96 s^2-96 s+35) \cos (x-y)-2 (2 s-1) (5 \cos (2 (x-y))+13)+\cos (3 (x-y))\big)^2
$$
which only depends upon $u:=x-y$ and $s$. Solving in $s$ gives two solutions
$$
s\to \frac{1}{4} (\cos (u)+2), \ \ s\to \frac{1}{12} (6 \cos (u)+\cos (2 u)+5) \sec (u).
$$
The second solution lies outside the interval $[0,1]$. The first solution singles out the critical set of $\beta$ as the two dimensional graph 
$$
G=\{(x,y,\frac{1}{4} (\cos (x-y)+2))\mid x,y \in \T\}.
$$
The critical values form a circle $C$ as obtained using the equality 
$$
\beta(x,y,\frac{1}{4} (\cos (x-y)+2))=(\cos (x),\sin (x),0,0).
$$
 Let $U$ be the range of the restriction of  $\beta$ to $\T^2\times (0,1)$. It is contained  in the complement of the singular set in the  hypersurface $K$, and it is connected (as the image of a connected set by a continuous map). It follows that $U$ is contained in a single component $V$ of the complement of the singular set in the  hypersurface $K$.  The intersection of $U$ with the complement $V\setminus C$ of $C$ in $V$ is both open and closed: it is open because the map $\beta$ is open at any element of the pre-image, it is closed since it agrees with the intersection of $\beta( \T^2\times[0,1])$ with $V\setminus C$.  Now since $C$ is a one dimensional circle and $V$ is three dimensional the  complement $V\setminus C$ of $C$ in $V$ is connected. It follows that $U=V$ since the circle $C$ is by construction in $U$. It remains to show that $V$ is equal to the boundary  $\partial S$. We know that $\partial S$ is homeomorphic to a sphere $S^3$ and that it is contained in the hypersurface $K$. The above proof shows that $V\subset \partial S$ and hence that the same holds for the closure of $V$, one has $\bar V\subset \partial S$.  Now the boundary  $\partial S$ is necessarily a union of closures of components of the complement of the singular set in the  hypersurface $K$. This follows from  the invariance of domain \cite{Bro21}.  Indeed  since the boundary $\partial S$ is a topological sphere contained in $K$, whenever a point $x \in  \partial S$ is a non-singular
point of $K$,  one can consider the injective continuous map from a small neighborhood of $x$ in  $\partial S$ to $K$ given by inclusion and, since a neighborhood of $x$ in $K$ is a standard ball one knows by the result
of Brouwer  that the image of the map is open. This suffices to show that the intersection of $\partial S$ with the complement of 
the singular set in $K$ is both closed and open and is thus a union of components. Now if  $\partial S$ contained another component than $V$ it would become disconnected after removing the boundary of $V$. But the boundary of $V$ is the M\"obius strip $\Sigma=\epsilon(\T^2)$ and  the complement of a M\"obius strip in the sphere $S^3$  is always connected. This shows that $V$ is equal to the boundary  $\partial S$ and completes the proof of the theorem.  
\endproof

\section{Toeplitz operator systems: general structure}
\label{sect:toeplitz}
This section contains a detailed analysis of the Toeplitz matrices, viewed as operator systems inside the matrices with complex coefficients. We will apply the techniques and concepts developed in Section \ref{sect:op-syst} and find that the rich structure found by direct computations in the previous section exists for general $n$. 

\subsection{Toeplitz matrices}
Let us start by giving the general definition of the Toeplitz operator system. 

\begin{defn}
  The operator system $\Toep{n}\subset M_n(\C)$ of $n \times n$ complex-valued {\em Toeplitz matrices} is defined by the vector space of matrices with constant diagonals, {\em i.e.} of the form
  $$
  T := \begin{pmatrix} t_{k-l} \end{pmatrix}_{kl}=\begin{pmatrix} t_0 & t_{-1} & \cdots &t_{-n+2} &t_{-n+1} \\
    t_1 & t_0 & t_{-1} & & t_{-n+2}\\
    \vdots & t_1 &t_0 &\ddots & \vdots\\
t_{n-2} & & \ddots & \ddots & t_{-1} \\
t_{n-1} & t_{n-2} & \cdots & t_1 & t_0\\
  \end{pmatrix} 
  $$
  with $t_k \in \C$ for $k =-n+1 ,\ldots, n-1$. 
  \end{defn}

\begin{prop}
  We have the following isomorphism of $C^*$-algebras:
$$
C^*_\env(\Toep{n} )\cong M_{n}(\C).
$$
Moreover, independently of $n$ one has ${\rm prop}(\Toep{n})=2$.
\end{prop}
\proof
Note that the Toeplitz matrices generate all complex matrices so that $M_{n}(\C)$ is a $C^*$-extension of $\Toep{n} $. Since $M_{n}(\C)$ is simple, it follows from Corollary \ref{corl:simple-env} that this extension is the $C^*$-envelope.

In order to compute the propagation number we use a basis $\{ \tau_j\}_{j=-n+1,\ldots, n-1}$ for the Toeplitz matrices given by $1$'s on the $j$'th diagonal and zeroes elsewhere, {\em i.e.} for positive $k$ we have 
$$
\tau_k = \sum_{i=1}^{n-k} e_{i,i+k}; \qquad \tau_{-k} = \sum_{i=1}^{n-k} e_{i+k,i}.
$$
It then follows that for any $k, p \geq 0$ we have 
$$
\tau_k \tau_{-p-k} = \sum_{j=1}^{n-p-k} e_{j+p,j}
$$
which is a matrix that only has non-zero entries on the $(-p)$'th diagonal: there it is equal to 1 except for the last $k$ entries of that diagonal (where it vanishes). Since $k$ and $p$ are arbitrary we get all matrix units $e_{ij}$ for $i \geq j$. Taking adjoints we also get the lower-diagonal matrix units and the proof is complete.
\endproof



\subsection{The Fej\'er--Riesz operator system}

We   consider functions on $S^1$ with only a finite number of non-zero Fourier coefficients. More precisely we define as follows the so-called {\em Fej\'er--Riesz operator system} $\CZ{n}$:
\begin{equation}
  \label{eq:CstZn}
  \CZ{n} = \left\{ a = (a_k)_{k \in \Z}: \text{supp}( a) \subset (-n,n)  \right\}.
\end{equation}
The elements in $\CZ{n}$ are thus given by sequences with finite support 
$$
a = (\ldots, 0 ,a_{-n+1},a_{-n+2},\ldots, a_{-1}, a_0, a_1, \ldots, a_{n-2},a_{n-1},0,\ldots)
$$
and this allows to view $\CZ{n}$ as an operator subsystem of the convolution $C^*$-algebra $C^*(\Z) \cong C(S^1)$. 
The adjoint $a \mapsto a^*$ is given by $a^*_k = \overline a_{-k}$. Note that self-adjoint elements in $\CZ{n}$ are thus given by so-called {\em palindromic} sequences for which $\overline a_k = a_{-k}$ for $|k| \leq n-1$. Also, an element $a$ is positive, $a \geq 0$, if and only if it is positive in $C^*(\Z) \cong C(S^1)$, \ie   $\sum_k a_k \zeta^k$ for $|\zeta| = 1$ defines a positive function on $S^1$.

\begin{prop}
  \begin{enumerate}
  \item The following defines an action of $S^1$ by complete order automorphisms of $\CZ{n}$:
    $$
(a_k) \mapsto (\lambda^k a_k ) ; \qquad (\lambda \in S^1).
    $$
  \item The \v{S}ilov boundary of the operator system $\CZ{n}$ is $S^1$.
  \item The $C^*$-envelope of $\CZ{n}$ is given by $C^*(\Z)$.
  \item The propagation number is infinite.
    \end{enumerate}
  \end{prop}
\proof
For (1) observe that for a positive $a$ we have that $\sum_k a_k \zeta^k$ is a positive function on $S^1$. But then the same applies to $\sum_k a_k (\zeta \lambda)^k$ as this amounts to a translation in $S^1$.

For (2) note that the natural $C^*$-extension is given by $C(S^1)$ which contains $\CZ{n}$ as a subsystem. For $n > 1$ one has $1, \cos(\theta), \sin(\theta) \in \CZ{n}$ so that the system separates points. In particular, there exists a function in $\CZ{n}$ attaining a unique maximum at some point in $S^1$. By translating this function over $S^1$ one remains in the operator system, while showing that for each point in $S^1$ there exists a function in $\CZ{n}$ that attains its unique maximum there. In other words, the \v{S}ilov boundary is all of $S^1$. But then also (3) follows by Proposition \ref{prop:silov}.

The final statement follows from the fact that sums of products of $k$ elements of $\CZ{n}$ all belong to $\CZ{nk}$ which is strictly smaller than $C^*(\Z)$. 
\endproof

The following old factorization result by Fej\'er \cite{Fej16} and Riesz \cite{Rie16}, plays a key role (and this is why we named the above operator system $\CZ{n}$ after them).
 
\begin{lma}
\label{lma:fr}
  Let $m \geq 0$ and let $I \subseteq [-m,m]$ be an interval of length $m+1$.
  Suppose that $p(z) = \sum_{k= -m}^{m} p_k z^k$ is a Laurent polynomial such that $p(\zeta) \geq 0$ for all $\zeta \in \C$ for which $|\zeta|=1$. Then there exists a Laurent polynomial $q(z)= \sum_{k\in I} q_k z^k$ so that $p(\zeta) =| q(\zeta)|^2$ for all $\zeta \in S^1 \subset \C $.
  \end{lma}
A proof can be found in \cite[Theorem 1.12]{GS58}. 

\begin{prop}
  \label{prop:CZn}
  Let $\CZ{n}$ be the operator system defined in Equation \eqref{eq:CstZn}.
  \begin{enumerate}
  \item The extreme rays in $(\CZ{n})_+$ are given by the elements $a\in(\CZ{n})_+$,  $a= (a_k)$ for which the Laurent series $\sum_k a_k z^k$ has all its zeroes on the circle $S^1$.     
  \item The pure states of $\CZ{n}$ are given by the functionals $a \mapsto \sum_k a_k \lambda^k$ for any $\lambda \in S^1$.
    \end{enumerate}
  \end{prop}
\proof
(1) Let $a\in (\CZ{n})_+$. Suppose that the Laurent series $P(z) = \sum_k a_k z^k$ has a zero $z_0$ such that $|z_0|<1$. Since $a^* = a$ this Laurent series also has a zero at $1/ \overline z_0$. We may thus factorize
$$
P(z) = (z-z_0)(1/z-\overline z_0) Q(z)
$$
for a Laurent series $Q(z)=\sum_k b_k z^k$ where $b_k=0$ for $\vert k\vert>n-2$. Since $(z-z_0)(\overline z-\overline z_0)$ is (strictly) positive when restricted to $|z|=1$, we also have that $Q(z)$ is positive when restricted to the circle. Moreover, there exists $\epsilon>0$ such that $(z-z_0)(\overline z-\overline z_0)-\epsilon \geq 0$ for $|z|=1$. Thus  $(z-z_0)(1/z-\overline z_0)=\epsilon + c$ is the sum of two elements of $(\CZ{2})_+$,  and $a=\epsilon b +cb$ is not extremal. We conclude that if $P(z)$ has a zero outside the circle, then $a$ is not extremal.

Suppose now that $P(z) = \sum_k a_k z^k$ has all its $2n-1$ zeroes on the circle. If
$b\in(\CZ{n})_+$ fulfills $b\leq a$, then $\sum_k b_k z^k=Q(z) \leq P(z)$ for $|z|=1$. Then, at a zero of $P(z)$ of multiplicity $k$, $Q(z)$ has a zero of  multiplicity at least equal to $k$.  This shows that $Q(z)$ is a scalar multiple of $P(z)$ and the proof of (1) is complete. 

For (2) we use that we can extend a pure state on $\CZ{n}$ to a pure state of $C(S^1)$, thus given by evaluation in a point of $S^1$. In view of the symmetry given by the action of $S^1$ on $\CZ{n}$ we conclude that all pure states of $C(S^1)$ restrict to pure states of $\CZ{n}$.
\endproof

\subsection{Duality with the Toeplitz operator system}
\label{sect:duality-toeplitz}
In this section we analyze a duality between the operator systems $\Toep{n}$ and $\CZ{n}$ for any $n \geq 0$. The advantage of this duality will become clear soon, when we analyze the pure state space and extremal positive elements in $\Toep{n}$.

\begin{prop}
\label{prop:toep}
  There is a complete order isomorphism between the operator system $\Toep{n}$
  and the dual of $\CZ{n}$. For any Toeplitz matrix $T= (t_{k-l})_{k,l}$
the functional $\phi_T \in (\CZ{n})^d$ is given by $\phi_T(a) = \sum_k a_k t_{-k}$ where $ a = (a_k) \in   \CZ{n}$.
\end{prop}
\proof
Since the vector space pairing given by the formula $\sum_k a_k t_{-k}$ is clearly non-degenerate, we simply have to check that $T \mapsto \phi_T$ respects the matrix-order and the order unit. Let $I=\{0, \ldots, n-1\}\subset (-n,n)$. In view of Lemma \ref{lma:fr}, an element $a \in \CZ{n}$ is positive if and only if it can be written as a convolution product $b^* \ast b$ for some  $b \in \CZ{I}$. One has, with $b_k=0$ $\forall k\notin I$,
$$
(b^* \ast b)(j)=\sum \overline b_k b_{j+k}   \qqq j, \vert j\vert <n
$$
where the summation takes place for $k\in I\cap (I-j)$. This intersection is non empty for $j\in (-n,n)$. One has using $j+k=l\Rightarrow j=l-k$
$$
   \phi_T(b^* \ast b)=\sum_{j\in (-n,n)}(b^* \ast b)(j)t_{-j}=\sum_{k,j}\overline b_k b_{j+k} t_{-j}=\sum_{k,l} \overline b_k b_l t_{k-l}=\langle b\vert Tb\rangle
$$
since $(T\xi)_k = \sum_{l}  t_{k-l}\xi_l$ for any $\xi \in \ell^2(I)$. Positivity of this expression is equivalent to the positivity of the Toeplitz form, {\em i.e.} $\phi_T \geq 0$ if and only if $T \geq0 $.

We show that for the order unit $1 \in \Toep{n}$ the functional $\phi_1$ is a faithful state on $\CZ{n}$. Since  for $a= b^* \ast b$ one has  $\phi_1(a)=\sum_k |b_k|^2$  the result follows.  
\endproof

This duality allows us to move smoothly between the following three structures:
\begin{enumerate}
\item a positive Toeplitz matrix $T \in \Toep{n}_+$;
\item a positive linear functional $\phi$ on $\CZ{n}$;
\item a positive quadratic form $Q$ on $C^*(\Z)_I$ of elements of sequences with support in an interval $I \subset \Z$ of length $n$.
  \end{enumerate}
In fact, these three structures are equivalent and related via the formulas:
\begin{align*}
&  1 \leftrightarrow 2 : \qquad \phi(a) = \sum t_k a_{-k} ;\\
&  2 \leftrightarrow 3 : \qquad \phi(\xi^* \ast \eta) = Q(\xi,\eta); \\
&  1 \leftrightarrow 3 : \qquad \langle \xi, T \eta \rangle = Q(\xi ,\eta).
  \end{align*}
With respect to these structures, we will be interested in $\ker \,T$, $\phi^\perp$ and both the radical and kernel of $Q$ where we recall that
\begin{align*}
  \phi^\perp &= \left\{ a \in \CZ{n} : \phi(a)  =0 \right\},\\
  \text{rad}(Q) &= \left \{\xi \in \CZ{I} : Q(\xi,\eta) = 0 \, \forall \eta \right\},\\
\ker(Q)  &= \left \{\xi \in \CZ{I} : Q(\xi,\xi) = 0 \, \forall \eta \right\}  .
\end{align*}
\begin{lma}
  \label{lma:properties-quadr-form}
  \begin{enumerate}
  \item For positive quadratic forms $Q$ the radical and kernel coincide. 
\item We have $(\phi^\perp)_+ = \left \{ \xi^* \ast \xi: \xi \in \ker \, Q \right \}$.
    \end{enumerate}
\end{lma}
\proof
(1) is a straightforward application of the Cauchy--Schwartz inequality.
For (2) note that if $a \in (\phi^\perp)_+$ then since $a \geq 0$ by Lemma \ref{lma:fr} it follows that $a = \xi^* \ast \xi$. But then $Q(\xi,\xi) = \phi(\xi^* \ast \xi) = \phi(a)=0$. The other inclusion is obvious.
\endproof
If no confusion can arise, we will also write $\phi^\perp_+ = (\phi^\perp)_+$ for the positive elements in the kernel of the linear functional $\phi$. 
\subsection{Pure states of the Toeplitz operator system}
We will determine the pure states of $\Toep{n}$, as well as the extreme rays in the cone $\Toep{n}_+$ of positive Toeplitz matrices. Here the duality with $\CZ{n}$ will turn out to be very useful, as it permits a simpler analysis and conceptual understanding of these extreme sets. We introduce the following notation:
$$
f_z = \frac 1 {\sqrt{n}}\begin{pmatrix} 1 & z & z^2 & \cdots & z^{n-1} \end{pmatrix}^t \in \C^n
$$
for any $z \in \C$. It is a column of a Vandermonde matrix ({\em cf.} Equation \eqref{eq:f-z} below).

\begin{prop}\label{proptoeplitzextremepoints}
  Let $\Toep{n}$ be the Toeplitz operator system. 
  \begin{enumerate}
  \item The extreme rays in $\Toep{n}_+$ are given by (multiples 
    of) $\gamma(\lambda) = \ket{f_{\lambda}}\bra{f_{\lambda}}$ for any $\lambda \in S^1$. In other words, the extreme rays $T= (t_{k-l})_{k,l}$ are of the form $t_k = \lambda^k$ (up to a positive real number) for some $|\lambda|=1$.
  \item The pure states of $\Toep{n+1}$ are given by functionals $T \mapsto \langle \xi, T \xi \rangle$ where the vector $\xi=(\xi_0,\ldots, \xi_n) \in \C^{n+1}$ is such that the polynomial $z \mapsto \sum_k \xi_k z^{n-k}$ has all its zeroes on $S^1$.
\item The pure state space $\cP(\Toep{n+1}) \cong \T^n /S_n$ is the quotient  of the $n$-torus by the symmetric group on $n$ objects. 
  \end{enumerate}
  \end{prop}
\proof
For the first two statements we use duality in the form of  Corollary \ref{corl:dual-extr} and Proposition \ref{prop:toep}. Proposition \ref{prop:CZn}  determines pure states and extreme rays in the dual system $(\CZ{n})_+$. The extreme rays in $\Toep{n}_+$ are given by pure states on $(\CZ{n})_+$ \ie by evaluation at points of $S^1$. Up to $\lambda\mapsto \lambda^{-1}=\bar \lambda $, they correspond to the $\gamma(\lambda)$. A pure state of $\Toep{n+1}$ corresponds to an extreme ray  in $(\CZ{n+1})_+$ and hence to  an  element $a\in(\CZ{n+1})_+$,  $a= (a_k)$ for which the Laurent series $\sum_k a_k z^k$ has all its zeroes on the circle $S^1$. Since $a \geq 0$ by Lemma \ref{lma:fr} it follows that $a = \xi^* \ast \xi$. Then the vector $\xi \in \C^{n+1}$ is such that the polynomial $z \mapsto \sum_k \xi_k z^{n-k}$ has all its zeroes on $S^1$.  For the third statement, let the zeroes of $\sum_k \xi_k z^{n-k}$ be labelled $\lambda_1, \ldots, \lambda_{n}$ (taken with multiplicities). Then, up to normalization, we can write $\xi$ in terms of elementary symmetric polynomials in the $\lambda_k$'s: 
\begin{equation}
  \label{eq:vector-v}
\xi =  \begin{pmatrix} 1 \\ \sum_k \lambda_k \\ \sum_{k < l} \lambda_k \lambda_l \\ \vdots \\ \lambda_1 \cdots \lambda_{n} \end{pmatrix},
\end{equation}
which gives, using (2) the required identification $\cP(\Toep{n+1}) \cong \T^n /S_n$.
\endproof

Note that this type of duality between cones of positive elements is central in the theory of matrix completion and moments, and appears for instance in \cite[Section 1.1]{BW11}. 

\medskip

As an example let us consider the case $\Toep{3}$. The description of extreme rays in Proposition \ref{proptoeplitzextremepoints}  agrees with Proposition \ref{extrepts}. The pure state space of $\Toep{3}$ is   given in Proposition \ref{proptoeplitzextremepoints} by vector states $\ket {\xi} \bra{\xi}$ with $\xi$ of the form
\begin{equation}
  \label{eq:toep-pure-3}
\xi = \frac1 {\sqrt{4 + 2 \cos (x-y)}} \begin{pmatrix} 1 \\ e^{ix} + e^{i y} \\ e^{i (x+y)} \end{pmatrix} ,
\end{equation}
where $x ,y \in [0,2\pi)$. This confirms the result from the previous section where we found in Theorem \ref{thmextremepoints} that this M\"obius strip is the pure state space of $\Toep{3}$ ({\em cf.} Figure \ref{surface}).

\subsection{The cone of positive Toeplitz matrices}
We now apply the above operator system duality to arrive at a description of the cone $\Toep{n}_+$. This generalizes the analysis done in Section \ref{sect:truncations} to arbitrary dimension.

As a first powerful application of the duality we derive a classical result of Carath\'eodory from 1911 \cite{Car11} stating that positive semi-definite Toeplitz matrices allow for a so-called Vandermonde factorization (see also \cite{AK62} and \cite[Chapter 4]{GS58}). More recently, the value of these kind of factorizations has been rediscovered in the context of 
  signal analysis ({\em cf.} \cite{Red84,Bac13,YX18} and \cite{BW11} for a mathematical treatment). But our main finding is   the 
extension to the general case of the peculiar properties of the hypersurface which determines the boundary of the cone $\Toep{n}_+$. As shown below in Theorem \ref{thmstratification} this hypersurface admits a remarkable stratification by the degree of singularity of its points and this stratification corresponds to the rank of positive Toeplitz matrices, thus extending the results of the special case $n=3$ to the general case.
We start with some preparation.
\begin{lma}
  \label{lma:cone-det0-n}
The cone $\Toep{n}_+$ of positive elements is the closure of the open component of the identity matrix in the complement of the hypersurface defined by $H := \{ T  \in \Toep{n}:  \det(T)= 0 \} $. In particular, the boundary of $\Toep{n}_+$ coincides with the boundary in the complement of $H$ of the component of $1$. 
    \end{lma}
  \proof
  The proof is completely analogous to Lemma \ref{lma:cone-det0} above but let us for convenience include it here for the general case. The cone is convex and is the closure of its interior which consists of matrices whose eigenvalues are strictly positive. The segment joining $T$ to the identity matrix $1$ stays inside $\Toep{n}_+$ and thus $T$ belongs to the open component of $1$ in the complement of the hypersurface $H$. Conversely on a path in the complement of this hypersurface joining $1$ to $T$ the eigenvalues remain positive since they vary continuously and cannot vanish as their product is given by the determinant. 
  \endproof

   A {\em face} $F$ of a convex cone $C\subset E$, in a real linear space, is a sub-cone $F\subset C$ which is  hereditary \ie 
  $$
  x\in F\ \text{ and } \ 0\leq_C y \leq_C x \Rightarrow y\in F.
  $$
  The intersection of $C$ with the real linear span $L(F)$ of $F$ is equal to $F$, since one has $L(F)=F-F$. Moreover if $C$ is proper, \ie $C\cap -C=\{0\}$ the projection $p(C)\subset E/L(F)$ of $C$ in the quotient is still proper.
  \begin{lma}
    Let $\phi$ be a positive linear functional on $\CZ{n}$. 
    \begin{enumerate}[a)]
    \item $\phi^\perp_+$ is a face of $(\CZ{n})_+$.
      \item $(\phi^\perp_+)_+^\perp$ is a face of $\Toep{n}_+$; it is the face generated by $\phi$. 
    \end{enumerate}
    \end{lma}
  \proof
  For (a) we suppose $a \in \phi^\perp_+$ and $b \leq a$ in $(\CZ{n})_+$. Then $\phi(b) \leq \phi(a) =0$ so $b \in  \phi^\perp_+$. Let us prove the second claim. Let $F$ be the face of $\phi$ and $L(F)=F-F$ its linear span. Then 
  $$a\in \phi^\perp_+\iff a\in (\CZ{n})_+ \cap L(F)^\perp  $$
 which is the dual of the projection of the cone in the quotient by $L(F)$. This projection is a proper cone, thus its dual is spanning and we get
  $$
  \psi \in (\phi^\perp_+)_+^\perp \iff \psi \in \Toep{n}_+\cap L(F)=F
  $$
  since the projection of $\psi$ is $0$ in the quotient by $L(F)$.
  \endproof

  \begin{prop}
The extreme rays of a face $F$ in a cone $K$ are extreme rays in $K$. 
    \end{prop}
  \proof
  Let $x$ be extreme in $F$ but suppose it is not extreme in $K$. Then there is an $y$ in $K$ such that $y \leq x$. Since $x$ is a point of a face, it follows that $y \in F$, and by extremality of $x$ in $F$ we find that $y=x$.
  \endproof

  \begin{thm}
Let $T$ be a rank $r \leq n-1$ positive Toeplitz matrix. Then the face generated by $T$ is a cone based on a simplex of dimension $r$ whose extreme points are the $\gamma(\lambda_i)$ where the $\lambda_1,\ldots, \lambda_r$ are the common roots of $\xi(\lambda)=0$ for all $\xi \in \ker\,T$.
  \end{thm}
  \proof
  Since $(\phi^\perp_+)^\perp_+$ is a face, it is generated by extreme rays $\gamma(\lambda)$ where the $\lambda$ are exactly the common zeroes of all $\xi \in \ker \, Q$. Let us denote these zeroes by $\lambda_1,\ldots, \lambda_m$ for $m \leq n-1$. By linear independence of the vectors $f_{\lambda_1},\ldots f_{\lambda_m}$ it follows that the $\gamma(\lambda_1), \ldots, \gamma(\lambda_m)$ are linearly independent. This implies that the face generated by $T$ is a cone based on an $r$-dimensional simplex with extreme points $\gamma(\lambda_1), \ldots,\gamma(\lambda_m)$. But then dually we must have $\dim(\phi^\perp) = n-m$ so that it follows that the rank of $T$ is equal to $m$. 
This completes the proof.
\endproof

We can reformulate this as the following Vandermonde factorization of positive Toeplitz matrices.
  \begin{corl}
\label{corl:toeplitz-decomp-rank}  Let $T$ be a positive $n \times n$ Toeplitz matrix of rank $r\leq n-1$. Then $T$ can be written in the following form:
  $$
T = V \Delta V^*,
  $$
where $\Delta$ is some positive diagonal matrix and $V$ is an $n \times r$ Vandermonde matrix,
  $$
\Delta= \begin{pmatrix} d_1 & &&& \\ & d_2 &&&\\ &&\ddots & \\ &&&&d_r\end{pmatrix}; \qquad  V=  \frac 1 {\sqrt n}  \begin{pmatrix} 
    1 & 1 & \cdots &1 \\
   \lambda_1 & \lambda_2 & \cdots & \lambda_n\\
    \vdots & & &\vdots \\
    \lambda_1^{n-1} & \lambda_2^{n-1} & \cdots & \lambda_{r}^{n-1}
 \end{pmatrix},
$$
for unique values $d_1, \ldots, d_r > 0$ and $\lambda_1,\ldots, \lambda_{r} \in S^1 \subset \C$.
    \end{corl}

\begin{thm}
\label{thm:toeplitz-decomp}
  Let $T$ be an $n \times n$ Toeplitz matrix of arbitrary rank. Then $T \geq 0$ if and only if $T$ is of the following form:
  $$
T = V \Delta V^*,
  $$
where $\Delta$ is some positive diagonal matrix and $V$ is a Vandermonde matrix,
  $$
\Delta= \begin{pmatrix} d_1 & &&& \\ & d_2 &&&\\ &&\ddots & \\ &&&&d_n \end{pmatrix}; \qquad  V=  \frac 1 {\sqrt n}    \begin{pmatrix} 
    1 & 1 & \cdots &1 \\
   \lambda_1 & \lambda_2 & \cdots & \lambda_n\\
    \vdots & & &\vdots \\
    \lambda_1^{n-1} & \lambda_2^{n-1} & \cdots & \lambda_n^{n-1}
 \end{pmatrix},
$$
for some $d_1, \ldots, d_n \geq 0$ and $\lambda_1,\ldots, \lambda_n \in S^1 \subset \C$.
\end{thm}
\proof
We take a base for the cone $\Toep{n}_+$ by fixing the trace of the Toeplitz matrices to be $n$. Note that this is a compact set. 

Let $T$ be a matrix of rank $n$ with trace $1$ and take an arbitrary extreme point $\gamma(\lambda)$. We consider a line segment from $\gamma(\lambda)$ to $T$ and prolong this segment until it reaches a point $T'$ on the boundary of the (compact) base of the cone. Since elements in the boundary of the cone of positive elements have vanishing determinant (Lemma \ref{lma:cone-det0-n}), the rank of $T'$ is $n-1$. Hence the above Theorem applies and we may write $T' = \sum_{k=1}^{n-1} d_k \gamma(\lambda_k)$ for some $d_k, \lambda_k$. Since $T = t T' + (1-t) \gamma(\lambda)$ for some $t \in (0,1)$ we may write $T = \sum_{k=1}^{n-1}t d_k \gamma(\lambda_k) +(1-t) \gamma(\lambda)$ and the proof is complete. 
\endproof

Given the above concrete realization of the extreme elements, we may wonder how they are related to the singular points of the hypersurface $H$ defined by $\det T =0$. In particular, we would like to generalize to arbitrary $n$ the results of the previous section where we found the extreme points of $\Toep{3}_+$ (with fixed trace) to coincide with the singular points on $H$.

Since $H$ is defined to be the zero-set of $\det T$ we may analyze the singular points by looking at the partial derivatives of $\det T$. The determinant of a matrix is a multilinear function of the entries of the matrix and the partial derivatives of any order with respect to the entries are given by the minors. When $\det T$ is evaluated on Toeplitz matrices it is no longer multilinear but we shall show that the singularities of $\det T$ are still related to the rank of $T$.  Let us denote the Fr\'echet derivative of a functional $f$ on the real vector space $\Toep{n}_h$ by $D^{(k)}(f)$; when evaluated at an element $T \in \Toep{n}_h$ it is a linear functional on $(\Toep{n}_h)^{\otimes k}$ defined by
$$
D^{(k)}(f)(T, T_1\otimes \cdots \otimes  T_k) =\frac{\partial}{\partial t_1}\cdots \frac{\partial}{\partial t_k} \det (T+t_1 T_1 + \cdots t_k T_k)|_{t_1=\cdots = t_k =0 },
$$
where $T_1, \cdots, T_k \in \Toep{n}_h$.

In the case at hand, there is a natural stratification of the determinant hypersurface given by the degree of vanishing of $\det T$
$$ \cdots \subset S_{k} \subset S_{k-1} \subset \cdots \subset S_1 \subset S_0 =  H
$$ 
where at level $k$ one imposes the many conditions
$$
S_{k} = \left\{ T \in H: D^{(k)}\det(T) = 0, D^{(k-1)}\det(T) = 0, \ldots,\det(T) = 0 \right\} .
$$
We will say that $T$ has {\em multiplicity} $k+1$ in the hypersurface $H$ if $T \in S_k$.

\begin{thm}\label{thmstratification}
In the boundary of the cone $\Toep{n}_+$ the stratification of the singular set of $H$ coincides with the stratification by the rank. More precisely, $T$ has multiplicity $m$ if and only if $T$ has rank $n-m$ for any $m = 0, \ldots n-1$.
  \end{thm}
\proof
Assume that $T$ has rank $\leq r$. Then with $q \leq n-1-r$ we have 
$$
\det \left(T + s_1 \gamma(\lambda_1) + \ldots + s_{q} \gamma(\lambda_{q} \right)) = 0
$$
for arbitrary $s_1, \ldots, s_{q} \geq 0$ and $\lambda_1,\ldots, \lambda_{q} \in S^1$. This implies that
$$
D^{(q)}(\det)(T,\gamma(\lambda_1) \otimes \cdots \otimes \gamma(\lambda_{q}) ) = 0
$$
for all $\lambda_1,\ldots, \lambda_q \in S^1$. Since by Theorem \ref{thm:toeplitz-decomp} the Toeplitz matrices are in the linear span of $\gamma(\lambda)$'s this implies that $D^{(q)}(\det)$ vanishes at $T$. Thus, $T$ has multiplicity $m$ with $m=n-r$.

In the other direction, for any $k$ let us suppose that $T$ has rank $r>n-1-k$ for some $k$. Then there are $s_1, \ldots ,s_k$ and $\lambda_1,\ldots, \lambda_k$ such that 
$$
\det\left (T + s_1 \gamma(\lambda_1) + \ldots + s_{q} \gamma(\lambda_{q} \right)  \neq 0.
$$
Since the $\gamma(\lambda_j)$ have rank one, this determinant is a polynomial of order $k$ in the $s_1,\ldots,s_k$ so that we find that $D^{(k)}(\det)(T) \neq 0$ for this $k$. Hence $T \notin S_k$ and the proof is complete. 
\endproof
\subsection{Distance on spectral truncations of the circle}\label{sect:trunc-dist}

We now compute the distance on the state space of $\Toep{n}$, using the formula 
$$
d(\phi,\psi):=\sup\{ \vert \phi(A)-\psi(A)\vert \mid \Vert [D,A]\Vert\leq 1\}.
$$
where $D$ is the Dirac operator on the circle ({\em cf.} Section \ref{sect:trunc-circle} above). We use only self-adjoint elements $A=A^*\in \Toep{n}_{\rm sa}$  in this formula. The distance is in fact determined by the following norm $\Vert  A\Vert_D$ on the quotient $\Toep{n}_{\rm sa}/\R 1$ of the  real vector space $\Toep{n}_{\rm sa}$ of Toeplitz selfadjoint matrices by the scalar ones:
$$
\Vert  A\Vert_D:=\Vert [D, A]\Vert.
$$  
More precisely one takes as the dual  of $\Toep{n}_{\rm sa}/\R 1$ the subspace $\CZ{n}^0$ of $(\CZ{n})_{\rm sa}$ given by linear forms which vanish on scalars. In the above formula one has $\phi-\psi\in \CZ{n}^0$, and the distance is determined by 
\begin{equation}\label{distdet}
	d(\phi,\psi)=\sup \{\vert (\phi-\psi)(A)\vert \mid \Vert  A\Vert_D\leq 1\},\end{equation}
so that, using the dual norm $\Vert  \bullet\Vert^D$ of $ \Vert  \bullet\Vert_D$, one gets  \begin{equation}\label{distdet1}
d(\phi,\psi)=\Vert \phi-\psi\Vert^D, \  \Vert  \omega\Vert^D:=\sup \vert \omega(A)\vert \mid \Vert  A\Vert_D\leq 1\qqq \omega \in \CZ{n}^0.
 \end{equation} 
The commutator $[D, A]$ is a Toeplitz matrix of trace $0$ and  one has a linear map $\partial$ from selfadjoint Toeplitz matrices to themselves given by 
\begin{equation}
\partial A:=i[D,A].
\label{eq:differential}
\end{equation}
Thus the unit ball  for the norm $\Vert  A\Vert_D$ is obtained by pulling back, by the map $\partial$  the unit ball of the Toeplitz norm in the subspace of elements of trace $0$. Now the latter is the intersection of two convex sets $C_\pm$ where 
$$
C_\pm:=\{A\in \Toep{n}_{\rm sa}\mid \tr(A)=0, \  1\pm A\geq 0\}.
$$ 
 The polar of a convex subset $C\subset E$ of a real vector space $E$ is defined as 
$$
C^o:=\{L\in E^*\mid L(\xi)\leq 1\qqq \xi \in C\}.
$$ 
 We have:
\begin{prop}\label{bipolar1}
$(i)$~The map $\partial$ gives an isomorphism $\partial:\Toep{n}_{\rm sa}/\R 1\to \Toep{n}_{\rm sa,0}$ (with Toeplitz matrices of trace $0$).\newline
$(ii)$~The transpose $\partial^t$ of $\partial$  is an isomorphism  $(\CZ{n})_{\rm sa}/\R 1\to \CZ{n}^0$.\newline
 $(iii)$~The unit ball of $ \CZ{n}^0$ for the norm $\Vert \bullet\Vert^D$ is the projection by the map  $\partial^t$ of the unit ball of the dual $\Vert\bullet\Vert^*$ of the Toeplitz operator norm.
 \newline
 $(iv)$~The unit ball of $ \CZ{n}^0$ for the norm $\Vert \bullet\Vert^D$ is the convex hull of the polars of  $\partial^{-1}C_\pm\subset\Toep{n}_{\rm sa}/\R 1$.
\end{prop}
\proof $(i)$~It is an isomorphism from the quotient by the kernel $\R 1$ to the range.\newline
$(ii)$~The transpose $\partial^t$ is similarly an isomorphism from the quotient by its kernel $\R 1$ with its range.\newline
 $(iii)$~The unit ball for the norm $\Vert \bullet\Vert^D$ is described as the subset of $ \CZ{n}^0$
 $$
 \Vert \omega\Vert^D\leq 1\iff \vert \omega(A)\vert\leq 1\qqq A \mid \Vert \partial A\Vert \leq 1.
 $$
 If $\omega=\partial^t(\psi)$ with $\Vert\psi\Vert^*\leq 1$ one gets, using $\omega(A)=\partial^t(\psi)(A)=\psi(\partial(A))$
 $$
 A, \Vert \partial A\Vert \leq 1\Rightarrow  \vert \omega(A)\vert=\vert \psi(\partial(A))\vert\leq 1.
 $$
 This shows that the projection by the map  $\partial^t$ of the unit ball of the dual $\Vert\bullet\Vert^*$ is contained in the unit ball of $ \CZ{n}^0$ for the norm $\Vert \bullet\Vert^D$. Conversely one can identify the dual of $\Toep{n}_{\rm sa,0}$ with $(\CZ{n})_{\rm sa}/\R 1$ since $1\in \CZ{n}$ pairs trivially with Toeplitz matrices with trace $0$. Then let $\omega\in  \CZ{n}^0$, $\Vert \omega\Vert^D\leq 1$. Let $\psi_0$ be the linear functional on $\Toep{n}_{\rm sa,0}$ uniquely defined by 
 $$
 \psi_0(\partial(A)):= \omega(A).
 $$
 Since $\Vert \omega\Vert^D\leq 1$ the norm of $\psi_0$, as a functional on a subspace of the normed space $\Toep{n}_{\rm sa}$, is $\leq 1$. Thus by Hahn--Banach it extends to an element $\psi$ with $\Vert\psi\Vert^*\leq 1$. Moreover one has $ \psi(\partial A)=\psi_0(\partial(A))=\omega(A)$, $\forall A$.\newline 
 $(iv)$~The unit ball of the Toeplitz norm in the subspace of elements of trace $0$ is the intersection of the two convex sets $C_\pm$ and thus its image by the inverse of the isomorphism $\partial$ is the intersection of the $\partial^{-1}C_\pm\subset\Toep{n}_{\rm sa}/\R 1$.
 We then use the general fact that for closed convex sets the polar of an intersection is the convex hull of the polars, as follows from the bipolar theorem.\endproof 
 We now determine the  polar of $\partial^{-1}C_\pm\subset\Toep{n}_{\rm sa}/\R 1$.
 \begin{lem}\label{bipolar2} $(i)$~An element  $\phi \in \CZ{n}^0$ 
belongs to the polar of $C_-$ if and only if the linear form $\tilde \phi=\phi +1 
$
belongs to the state space $\cS$ of $\Toep{n}$. \newline
$(ii)$~The polar of $\partial^{-1}C_-\subset\Toep{n}_{\rm sa}/\R 1$ is $ \partial^t \cS$.
\end{lem}
 \proof $(i)$~One has $\tilde \phi(1)=1$ by construction and $\tilde \phi(A)=\phi(A)$ for any $A\in \Toep{n}_{\rm sa,0}$. Moreover for such an $A$ one has  $A\in C_- \iff 1-A \geq 0$ and thus
$$
\left(\phi(A)\leq 1\qqq A\in C_-\right)\iff \left(\tilde \phi(1-A)\geq 0\qqq A\in C_-\right)\iff \tilde \phi\in \cS
$$
since the elements of the form $1-A, A\in C_-$ are the positive Toeplitz matrices of fixed trace $=n$ and form a base of the positive cone $\Toep{n}_+$. \newline
$(ii)$~Given an isomorphism $T:E\to F$ of finite dimensional real vector spaces and a subset $X\subset E$, the polar of $T(X)\subset F$ is the image by the inverse of $T^t$ of the polar of $X$. Applying this to the isomorphism $\partial^{-1}:\Toep{n}_{\rm sa,0}\to \Toep{n}_{\rm sa}/\R 1$ one gets that the polar of $\partial^{-1}C_-\subset\Toep{n}_{\rm sa}/\R 1$ is the image by $\partial^t$ of the polar of $C_-$ and by $(i)$ one obtains $ \partial^t \cS$. \endproof 
One passes from $C_-$ to $C_+$ by multiplication by $-1$ and the same holds for the polars. Thus Lemma \ref{bipolar2} also determines the polar of $\partial^{-1}C_+$ as $-\partial^{t}\cS$. 
\begin{prop}\label{bipolar3} The unit ball of $ \CZ{n}^0$ for the norm $\Vert \bullet\Vert^D$ is the convex hull of   $\partial^{t}\cS$ and $-\partial^{t}\cS$.	
\end{prop}
\proof This follows from $(iv)$ of Proposition \ref{bipolar1} and Lemma \ref{bipolar2}.\endproof 
Coming back to Proposition \ref{bipolar1} $(iii)$, note that it is not true that the dual $\Vert\bullet\Vert^*$ of the Toeplitz operator norm is the operator norm in $ \CZ{n}$. In fact in the limit $n\to \infty$ the Toeplitz operator norm becomes the $L^\infty$ norm and its dual is the $L^1$ norm. This suggests to compare the norm $\Vert \bullet\Vert^D$ with the image by $\partial^{t}$ of the  quotient norm of the $L^1$ norm in $(\CZ{n})_{\rm sa}/\R 1$. We show that the  norm $\Vert \bullet \Vert^*$ on  $ \CZ{n}$ dual to the Toeplitz operator norm is larger than the $L^1$ norm for the normalized Haar measure $\frac{\vert dz\vert}{2\pi}$ on $S^1$.
\begin{prop}\label{ell1compare} $(i)$~Let $a= (a_k)\in (\CZ{n})_{\rm sa}$ with Laurent series $a(z):=\sum_k a_k z^k$, then one has 
\begin{equation}
\Vert a \Vert^*\geq \frac{1}{2\pi} \int_{S^1}\vert a(z)\vert \, \vert dz\vert= \Vert a \Vert_1 .
\label{normsell1}
\end{equation}
$(ii)$~The norm $\Vert \bullet\Vert^D$ fulfills the inequality
\begin{equation}
\Vert a \Vert_D\geq \inf_{\partial^t(b)=a} \Vert b \Vert_1 .
\label{norms2}
\end{equation}
\end{prop}
\proof $(i)$~The unit ball for the Toeplitz operator norm in $\Toep{n}_{\rm sa}$ is the interval $[-1,1]$ \ie intersection of $1-\Toep{n}_+$ and $-1+\Toep{n}_+$, thus its polar is the convex hull in $(\CZ{n})_{\rm sa}$ of the polars and it suffices to show that each is contained in the unit ball of the norm $\Vert a \Vert_1$. One has 
$$
\phi(1-\Toep{n}_+)\leq 1\Rightarrow \phi\geq 0\, \text{ and } \, \phi(1)\leq 1
$$
and since the positive elements of $(\CZ{n})_{\rm sa}$ are positive functions on $S^1$ (\ie the associated Laurent polynomial is positive) the $L^1$ norm is simply the integral and the latter is $\phi(1)\leq 1$. By symmetry the  polar of $-1+\Toep{n}_+$ is also contained in the unit ball of the $L^1$ norm.\newline
$(ii)$~The norm $\Vert \bullet\Vert^D$ is, by Proposition \ref{bipolar1} $(iii)$ the image of the norm $\Vert \bullet \Vert^*$ by the projection associated to the map $\partial$ thus the statement follows from $(i)$.\endproof  

Note the infimum which appears in formula \eqref{norms2}. It is directly related to the fact that the geodesic distance is computed using the shortest path between two points. More precisely we take $n=\infty$ and consider the distance between two points $x,y$ of the circle incarnated as the associated Dirac masses $\delta_x,\delta_y$ viewed as states. Then the choice of an element $a$ such that $\partial a=\delta_x -\delta_y$ is unique up to the addition of a constant. It contains $\pm$ times the characteristic function of the two intervals joining $x$ and $y$ as well as  affine combinations of these two solutions. One finds that the infimum taken in formula \eqref{norms2}, where we use the $L^1$ norm in the limit $n=\infty$, corresponds to the choice of the shortest interval. Moreover the $L^1$ norm of the characteristic function of the  interval is the distance between  $x$ and $y$. To go further in the exploration of the distance function on the truncated circle involves understanding how the state space converges to the space of probability measures on $S^1$ and the distance to the Kantorovich metric.
By Theorem 3.7 of  \cite{Cabrelli}, the Kantorovich distance $d_T(\mu,\nu)$ between two probability measures $\mu,\nu$  on $S^1$ is computed by the formula 
\begin{equation}\label{knownkantorov}
d_T(\mu,\nu)=\int_{S^1}\vert \alpha(x)-a \vert dx, \ \ \alpha(x)=\mu([0,x])-\nu([0,x])
\end{equation}
and where the constant $a$ is such that the integral is minimal. The derivatives of the functions $\mu([0,x])$ and $\nu([0,x])$ give $\mu$ and $\nu$ and we see that  Proposition \ref{bipolar1} $(iii)$ is the  version of the above formula for $d_T$ for the truncated circle. Thus we obtain
\begin{thm}\label{thmtruncatedd} $(i)$~The distance function on the state space $\cS$ of the operator system $\Toep{n}$ is given for $\phi, \psi \in \cS$ and primitives $\Phi, \partial^t\Phi=\phi$, $\Psi, \partial^t\Psi=\psi$ by
$$
d(\phi,\psi)= \inf_{c\in \R} \Vert \Psi -\Phi -c \Vert^*.
$$
$(ii)$~The distance function on the state space $\cS$ is larger than the Kantorovich distance $d_T(\phi,\psi)$ of the associated probability measures on the circle.	
\end{thm}
\proof $(i)$~Follows from \eqref{distdet1} combined with Proposition \ref{bipolar1} $(iii)$.\newline
$(ii)$~The inequality \eqref{normsell1}  gives 
$$
\Vert \Psi -\Phi -c \Vert^*\geq \Vert \Psi -\Phi -c \Vert_1
$$ 
combining with $(i)$ one gets 
$$
d(\phi,\psi)= \inf_{c\in \R} \Vert \Psi -\Phi -c \Vert^*\geq \inf_{c\in \R}\Vert \Psi -\Phi -c \Vert_1.
$$
Thus the result follows from  \eqref{knownkantorov} which shows that the last term is $d_T(\phi,\psi)$.\endproof 

These results relate Connes' distance formula for the truncated system to the explicit integral formula for the Kantorovich distance. More general results relating the distance on a spectral truncation of a given geometry to the Kantorovich distance on probability measures can be obtained from \cite[Proposition 3.6]{ALM14}. Indeed, from its very definition it is clear that Connes' distance formula for a Riemannian spin geometry coincides with Kantorovich's formulation of the distance formula on the state space of probability measures. Note that this formed the basis for the development of compact quantum metric spaces \cite{Rie99} (see also \cite{AM10} for a nice overview of the relation between the relevant distance functions).

\section{Toeplitz and circulant matrices}
\label{sect:circulant}
There is an interesting relation between the Toeplitz operator system $\Toep{n}$ discussed in the previous section and the group algebra of the cyclic group $C_{m}$ of order $m =2n-1$. We first recall the structure of this group algebra and the finite Fourier transform and formulate it in terms of operator systems.

\subsection{Fourier transform on the cyclic group of order $m$}
Let $m > 0$ and consider the finite abelian group $C_m := \Z/m \Z$. The point-wise action of $l^\infty(C_m)$ on $l^2(C_m)$ is given by:
\begin{equation}
\label{eq:action-cyclic}
g \cdot \psi (k) = g(k) \psi(k); \qquad (g \in l^\infty(C_m), \psi \in l^2(C_m), k \in C_m).
\end{equation}
In terms of the standard basis of $l^2(C_m)$ this becomes matrix-multiplication by a diagonal matrix $\text{diag}(g(0), \ldots, g(m-1) )$.

The {\em finite Fourier transform} is a map $\F: l^2(C_{m}) \to l^2(C_{m})$ defined by 
$$
\F(\psi)(k) = \sum_{l=0}^{m-1} \psi(l) \bar\zeta^{kl}
$$
with $\zeta$ a primitive $m$'th root of unity. The inverse finite Fourier transform $\bar\F$ is given by the same formula with $\bar\zeta$ replaced by $\zeta$. It is well-known that  $m^{-1/2}\F$ is unitary (Plancherel) so that $\F\bar \F=\bar \F\F=m$ and that both $\F$  and $\bar \F$ replace the convolution  product
  $$
  f\star g (k):=\sum_{l=0}^{m-1} f(k-l) g(l) 
$$ 
by the point-wise product since the Haar measures used to define them and $*$ are the same. Thus the Fourier transform is an isomorphism of the group algebra $C^*(C_m)$ with $l^\infty(C_m)$, $\F (f \ast  g)=\F f\cdot \F g$.  The unitary $U=m^{-1/2}\F$ conjugates the above representation \eqref{eq:action-cyclic} of $l^\infty(C_m)$ with the action of the group algebra $C^*(C_m)$ on $l^2(C_m)$ by convolution  product 
$$
f \ast  \psi=U^*(\F (f) \cdot U\psi)\qqq f\in  C^*(C_m), \ \psi \in l^2(C_m).
$$
The action by convolution of an element $c = (c_l) \in C^*(C_m)$  in terms of the standard basis of $l^2(C_m)$ is the following matrix acting on column vectors
\begin{equation}
\label{eq:circulant}
 c \sim \begin{pmatrix} c_0 & c_{m-1} & \cdots & c_2 & c_1 \\ c_1 &c_0 & c_{m-1} & & c_2 \\ \vdots & c_1 & c_0 & \ddots & \vdots \\ c_{m-2}  && \ddots & \ddots &c_{m-1} \\ c_{m-1} & c_{m-2} & \cdots & c_1 & c_0 \end{pmatrix}.
\end{equation}
Such a matrix is called a {\em circulant matrix}, it is a special case of a Toeplitz matrix. 

If we write as before
\begin{equation}
  \label{eq:f-z}
f_z  =\frac 1 {\sqrt{m}} \begin{pmatrix} 1 \\ z \\ \vdots \\ z^{m-1} \end{pmatrix}; \qquad z \in \C,
\end{equation}
then the Fourier transform can be written in terms of the canonical basis of $l^2(C_m)$ as a  Vandermonde matrix which is $\F=m^{1/2}U$ with $U$ unitary
$$
\mathcal F = \begin{pmatrix} 1 & 1 &\cdots & 1 \\ 1 & \xi & \cdots & \xi^{m-1} \\ \vdots & \vdots & \vdots & \vdots \\ 1 & \xi^{m-1} & \cdots & \xi^{(m-1)(m-1)}\end{pmatrix},\qquad U=  \begin{pmatrix} f_{1} & f_\xi & \cdots & f_{\xi^{m-1}} \end{pmatrix}
$$
with $\xi=\bar \zeta$ a primitive $m$'th root of unity. Consequently, $U=m^{-1/2}\F$ is the transformation matrix that diagonalizes the above circulant matrix \eqref{eq:circulant}. 

The finite Fourier transform can be understood nicely in terms of a duality of finite-dimensional operator systems, very similar to Proposition \ref{prop:toep} above. 
\begin{prop}
$(i)$~The operator system $C^*(C_m)$ is its  dual under the pairing  
$$
C^*(C_m) \times C^*(C_m) \to \C,   \   \
\langle f , g \rangle_{C^*(C_m)} := (f\star g)(0)=\sum_l f_l\, g_{-l}.
$$
$(ii)$~The following pairing gives a duality between $C^*(C_m)$ and $l^\infty(C_m)$
\begin{align*}
C^*(C_m) \times l^\infty(C_m) &\to \C\\
\left( c , g \right) &\mapsto \langle c ,  \F g \rangle_{C^*(C_m)}= \sum_{l,k} c_l \,  g(k) \zeta^{kl}.
\end{align*}
\end{prop}
\proof $(i)$~The Fourier transform $\F$ is an isomorphism of the operator system $C^*(C_m)$ with $l^\infty(C_m)$ and the latter system is its own dual under the pairing 
$
\langle h ,   g \rangle_{l^\infty}:=\sum_l h(l) g(l)
$.
Moreover one has for $f,g\in C^*(C_m)$, by Fourier inversion
$$
\langle f , g \rangle_{C^*(C_m)}=(f\star g)(0)=\frac 1m \bar \F(\F(f \star g))(0)=\frac 1m \bar \F(\F(f)\cdot \F(g))(0)=$$ $$=\frac 1m \sum_l \F(f)(l) \F(g)(l)=\frac 1m \langle \F(f) ,   \F(g) \rangle_{l^\infty}.
$$
Thus the isomorphism $\F:C^*(C_m)\to l^\infty(C_m)$ is compatible (up to normalization) with the pairing, so $(i)$ follows. \newline
$(ii)$~Follows from $(i)$. Note that both $\F$ and $\bar \F$ are isomorphisms $C^*(C_m)\to l^\infty(C_m)$ 
so that there is no issue on the choice of $\F$ in the formula for the pairing.
\endproof

The operator norm of elements of $C^*(C_m)$ is given by the sup norm ($l^\infty$ norm) of the Fourier transform. The dual of the operator norm is given exactly by the $l^1$ norm and both norms are easier to compute than for the Toeplitz operator system and its dual. 

This duality implies that there is a one-to-one correspondence between pure states of $l^\infty(C_m)$ and extreme rays in the positive cone $C^*(C_m)_+$, as well as the converse. Of course, in all cases, these spaces are just given by the $m$'th roots of unity. Note that viewing $C^*(C_m)$ as a subsystem of $\Toep{m}$ as shown in \eqref{eq:circulant} the extreme rays of $C^*(C_m)_+$ are those extreme rays of $\Toep{m}_+$ which belong to $C^*(C_m)$.

It is interesting to compare the finite nature of this structure  to the much richer structure encountered for the Toeplitz operator system, where a whole $S^1$-worth of extreme rays has been found, not to mention the rich structure of the pure state space. The underlying reason is that of symmetry: for $\Toep{n}$ the symmetry group is $S^1$ while for the circulant matrices this is reduced to the cyclic group of order $m$. Moreover, a comparison between the frameworks of the cyclic group and the Toeplitz operator system suggests that Theorem \ref{thm:toeplitz-decomp} is a generalization of the (finite) Fourier transform. In the next subsection we will further explore the relation between circulant and Toeplitz matrices.

\subsection{Relation between circulant and Toeplitz matrices}
Given a Toeplitz matrix of size $n$, it is possible to `complete' it to a circulant matrix of size $2n-1$. More precisely, we have the following classical result. 

\begin{prop}
Let $m \geq 2n-1$. Then any Toeplitz matrix $T \in \Toep{n}$ can be obtained as the compression of an $m \times m$ circulant matrix $C$ to the upper-left $n \times n$ corner:
  $$
T = P_n C P_n
$$
where $P_n$ projects onto the linear span of the first $n$ canonical basis vectors. In other words, conjugation by $P$ induces a completely positive map $C^*(C_{m}) \to \Toep{n}$. 
  \end{prop}
\proof
For any Toeplitz matrix $T= (t_{k-l})$, the sought-for circulant matrix is given as follows:
$$
C= \left( \begin{array}{ccccc|cccc}
  t_0 & t_{-1} & \cdots &t_{-n+2} &t_{-n+1} &  t_ {n-1} & \cdots & t_2 & t_1 \\
    t_1 & t_0 & t_{-1} & & t_{-n+2} & t_{-n+1} &&\cdots &t_2  \\
    \vdots & t_1 &t_0 &\ddots & \vdots & & \ddots & & \vdots\\
t_{n-2} & & \ddots & \ddots & t_{-1} & &\ddots  & \ddots & t_{n-1} \\
t_{n-1} & t_{n-2} & \cdots & t_1 & t_0 & t_{-1} & \cdots & & t_{-n+1}\\
\hline
t_{-n+1}& t_{n-1} & t_{n-2}&\cdots & t_1& t_0 & t_{-1} & \cdots & t_{-n+2}  \\
\vdots & & \ddots &\ddots & & \ddots & \ddots& & \vdots   \\
t_{-2} && &  \ddots &\ddots && \ddots & \ddots&  t_{-1} \\
t_{-1} & t_{-2} & \cdots & t_{-n+1} & t_{n-1} & t_{n-2} & \cdots & t_1 & t_0
  \end{array} \right).
$$
\endproof

Another intriguing property of the inclusion of the Toeplitz operator system inside $M_{2n-1}(\C)$ by the map $T\mapsto T \oplus 0_{n-1}$ is the following construction of a map
$$
\phi :l^\infty(C_{2n-1}) \otimes \Toep{n} \to M_{2n-1}(\C),
$$
where both $\Toep{n}$ and $M_{2n-1}(\C)$ are equipped with the operator norm. 

We let $S$ denote the (cyclic) shift matrix in $M_{2n-1}(\C)$ defined by $S(e_k) = e_{k+1}$ for $k=1,\ldots, 2n-2$ and $S(e_{2n-1}) = e_1$. We then define in terms of $f \in l^\infty(C_{2n+1})$ and 
a Toeplitz matrix $T \in \Toep{n}$:
$$
\phi(f \otimes T ) = \sum_{k=1}^{2n-1} f_k S^k \left(T \oplus 0_{n-1} \right) S^{-k} \in M_{2n-1}(\C).
$$
Since all operator spaces are finite-dimensional, $\phi$ is completely bounded but bijectivity of it is harder to establish and, in fact, not always true. 
\begin{prop}
  If $2n-1$ is a prime number, then the map $\phi$ is a completely positive bijection on the minimal tensor product
  $$
l^\infty(C_{2n-1}) \otimes_\min \Toep{n} \cong M_{2n-1}(\C).
  $$
\end{prop}

\proof To prove the proposition it is enough to show that the map $\phi$ is completely positive and that it is surjective, since the dimensions of both sides are the same.  

Note that the minimal tensor product is realized as an operator system inside $\B(l^2(C_{2n-1}) \otimes \C^n)$. In other words, if we write $f = \sum f_k \delta_k\in l^\infty(C_{2n-1})$ then a general element $\sum_k f_k \delta_k \otimes T_k  \in l^\infty(C_{2n-1}) \otimes_\min \Toep{n}$ is realized as the Kronecker product $\text{diag}(f_1 T_1, \ldots, f_{2n-1} T_{2n-1})$. 
If this is a positive matrix then the image under $\phi$ is clearly positive.

For surjectivity of $\phi$ one first identifies the matrix algebra with the crossed product of $\ell^\infty(C_{2n-1})$ represented as diagonal matrices by the action of $C_{2n-1}$. More precisely, any element  $x\in M_{2n-1}(\C)$ is uniquely of the form 
$$
x=\sum_{k=1}^{2n-1} f_kS^k, \ f_k\in \ell^\infty(C_{2n-1}).
$$ 
Note that the sub-spaces of the crossed product of $\ell^\infty(C_{2n-1})$ by the action of $C_{2n-1}$ involving a fixed power $j$ of $S$ are pairwise linearly independent and span the whole crossed product so that it suffices to show the surjectivity of $\phi$ on each such subspace.
The Toeplitz operators are uniquely obtained as compressions on the  $n$ dimensional space with projection $P=1_{n}\oplus 0_{n-1}\in \ell^\infty(C_{2n-1})$ of linear combinations of powers of $S$. Let 
 $\tau_j=PS^jP $ be the Toeplitz operator obtained by compression of $S^j$ on  $P$. These operators  fulfill, for fixed $j$,
$$
\phi(\delta_k \otimes \tau_j)= S^k \left( \tau_j\oplus 0_{n-1} \right) S^{-k}=S^k PS^jP S^{-k}=S^kS^j Q S^{-k}=S^jS^k Q S^{-k}
$$
where $Q=S^{-j}PS^{j}P$ is a fixed self-adjoint idempotent. It is non-zero since it is given by the intersection of two sub-spaces of dimension $n$ and the sum of dimensions exceeds $p=2n-1$. Thus, fixing $j$ and taking linear combinations of the form
$$
\phi\left(\sum_{k=1}^{2n-1} \delta_k \otimes \lambda_k \tau_j\right) =
S^j\sum_{k=1}^{2n-1}\lambda_k\,S^k Q S^{-k},
$$
the next  lemma applies and gives the required surjectivity of the linear map $\phi$.
\endproof
\begin{lma}Let $C_{2n-1}$ be a cyclic group of prime order $p=2n-1$. \newline 
$(i)$~Let $X\subset C_{2n-1}$ be a non-empty subset $X\neq C_{2n-1}$ and $\chi\in \hat C_{2n-1}$ a  character of $C_{2n-1}$. Then $\sum_{g\in X} \chi(g)\neq 0$.\newline
$(ii)$~Let $0<Q<1$ be a self-adjoint idempotent in $\ell^\infty(C_{2n-1})$. Then the linear space generated by the  conjugates $ S^{k} Q  S^{-k} $ (under the action of $C_{2n-1}$ on itself) is $\ell^\infty(C_{2n-1})$.	
\end{lma}
\proof $(i)$~We can assume that $\chi$ is non-trivial. Since $p$ is prime the subgroup $\chi(C_{2n-1})\subset \mu_p$ (roots of unity of order $p$) is equal to $\mu_p$. Thus it is enough to show that for any subset  $Y\subset\mu_p$, $Y\neq \emptyset$ such that $\sum_{u\in Y}u=0$ one has $Y=\mu_p$. Let $\xi$ be a primitive root of $1$ of order $p$ and $Z\subset \{0,\ldots, p-1\}$ such that $Y=\xi^Z$. Then the polynomial $A(x):=\sum_Z x^j$ fulfills $A(\xi)=0$ and is hence a multiple of the cyclotomic polynomial. But the latter is of degree $p-1$ since $p$ is prime, and thus one gets that $A$ is equal to the cyclotomic polynomial and thus $Z=\{0,\ldots, p-1\}$.\newline
$(ii)$~Let $X\subset C_{2n-1}$ be the non-empty subset $X\neq C_{2n-1}$ corresponding to the self-adjoint idempotent $0<Q<1$. Let $E$ be the linear space generated by the  conjugates $ S^{k} Q  S^{-k} $. The invariance of $E$ under the action of $C_{2n-1}$ means that its image $\hat E$ under Fourier transform is an ideal. Thus if it is non-trivial there exists a point of the dual group $\hat C_{2n-1}$ on which all elements of  $\hat E$ vanish. Equivalently this means that there exists a character $\chi$ of $C_{2n-1}$ such that $\langle \chi,E\rangle=0$ or equivalently that $\langle \chi,Q\rangle=0$. One has
$$
\langle \chi,Q \rangle=\sum_{g\in X} \chi(g).
$$
Thus by $(i)$ one gets that  $\langle \chi,Q\rangle \neq 0$ and this shows that $E=\ell^\infty(C_{2n-1})$.	\endproof

Note that the map $\phi$ is not a complete order isomorphism. This can be seen as follows. Note that up to a scaling factor it is a unital map, so by Proposition \ref{prop:rel-maps-opsyst} it is a complete order injection if and only if it is a complete isometry. But a simple calculation for $2\times 2$ Toeplitz matrices already shows that $\phi$ is not isometric.

\section{Outlook}\label{sect:outlook}
In this paper we have introduced a new approach to noncommutative geometry where the prominent role traditionally played by $C^*$-algebras is taken over by operator systems. The matrix ordering makes it possible that most of the theory still goes through, including state spaces, cones of positive elements, distance functions, etc. 

The examples we have considered show that spectral truncations allows one to work with finite-dimensional operator systems, while keeping in tact the full symmetry of the original space. For instance, the Toeplitz operator systems possess an $S^1$ symmetry, and as a consequence have a very rich extremal and pure state space structure. This is in contrast with the circulant matrices, where the symmetry is reduced to a discrete group. So, even though both spaces converge in Gromov--Hausdorff distance to the circle, for the second one loses a lot of structure in the finite-dimensional reduction.

The duality between the Toeplitz operator system $\Toep{n}$ and the (truncated) group algebra $\CZ{n}$ also uncovers the following intriguing relation between the fermionic and bosonic content of a spectral triple. As explained at the beginning of Section \ref{sect:truncations} the truncation on the Fourier modes of the (fermionic) vectors in the Hilbert space gives rise to the Toeplitz operator system $\Toep{n} = P_n C(S^1) P_n$, but the dual system $\CZ{n}$ describes truncations of the Fourier modes of the (bosonic) elements in the function algebra $C(S^1)$.

In a forthcoming paper  we shall show how, using as proposed in this paper operator systems rather than $C^*$-algebras, the fundamental idea of noncommutative geometry of associating a noncommutative $C^*$-algebra to a quotient space which is intractable by standard topological methods, extends to situations where the equivalence relation defining the quotient is no longer assumed to be transitive.  Such relations are called {\em tolerance relations} and can be traced back to Poincar\'e in his {\em Science and Hypothesis} (though the name was coined in \cite{Zee62}). Poincar\'e argued that in the physical continuum (in contrast with the mathematical continuum) it can hold that for measured quantities one has $A=B$, $B=C$ while $A<C$ due to potentially added measurement errors (see \cite{Sos86} for a development of the mathematical theory). This will allow us to extend the scope of noncommutative geometry and, in particular, to   introduce another operator system that appears naturally when one studies spaces up to some energy scale. In terms of position space this amounts to introducing a finite resolution $\eps$ and the tolerance relation between  points $x,y$ which is given by $d(x,y) < \eps$.  It allows one to define an operator system which, in the case that the relation is transitive, becomes the usual equivalence groupoid $C^*$-algebra \cite{Ren80}. We shall analyze the $C^*$-envelope and express the propagation number in terms of the diameter of the metric space. We will also characterize the pure state space of the operator system by means of a support condition on vector states.

\newcommand{\noopsort}[1]{}\def\cprime{$'$}

\end{document}